\documentclass[11pt,draft,dvips]{amsart}  
\usepackage{amssymb,amsmath,amsfonts}  
\usepackage[dvips]{graphicx,color}
\usepackage{cite}
\allowdisplaybreaks[1]
\definecolor{pgray}{gray}{0.8}

\newtheorem{theorem}{Theorem}[section]

\newtheorem{proposition}[theorem]{Proposition}

\newtheorem{remark}[theorem]{Remark}
\numberwithin{equation}{section}
\title{\mbox{}}

\begin{document}
%\begin{flushright}
%	2.26.2018
%\end{flushright}
\begin{center}
{\bf \LARGE{
	Large-time Behavior and Far Field Asymptotics of Solutions to the Navier-Stokes Equations
%	Asymptotic behaviour of solutions to a parabolic system of semi-conductor model
}}\\
\vspace{5mm}
\vspace{5mm}
{\sc\large
	Masakazu Yamamoto}\footnote{Graduate School of Science and Technology,
%Graduate School of Science,
	Niigata University,
	Niigata 950-2181, Japan}\qquad
\end{center}
\maketitle
\vspace{-15mm}
\begin{abstract}
Asymptotic expansions of global solutions to the incompressible Navier-Stokes equation as $t$ tends to infinity with high-order is studied and large-time behavior of the expansion is clarified.
Furthermore, far field asymptotics also is derived.
Those expansions are provided without moment conditions on the initial velocity.
The Biot-Savard law together with the renormalization for the vorticity equations yields those expansions.
\end{abstract}
%%% ----------------------------------------------------------------------
%
%% main text

\section{Introduction}
We consider decay properties of solutions to the incompressible Navier-Stokes equations in $\mathbb{R}^n$.
In several preceding works, asymptotic expansion of a solution is provided.
Those expansions require the fast decay for an initial velocity.
On the other hand, it is well known that a solution to the Navier-Stokes equation has a slow decay-rate as $|x| \to +\infty$.
This special structure of the Navier-Stokes equation disturbs to derive the asymptotic expansion with high-order.
In this paper, we investigate the asymptotic expansion with high-order without the strong assumption for an initial velocity as $|x| \to +\infty$.
Here we study the following initial-value problem:
\begin{equation}\label{NS}
\left\{
\begin{array}{lr}
	\partial_t u + (u\cdot\nabla) u = \Delta u - \nabla p,
	&
	t>0,~ x \in \mathbb{R}^n,\\
	\nabla\cdot u = 0,
	&
	t>0,~ x \in \mathbb{R}^n,\\
	u(0,x) = a (x),
	&
	x \in \mathbb{R}^n,
\end{array}
\right.
\end{equation}
where $n \ge 2$ and $a = (a_1,\ldots,a_n)$ is an initial velocity.
Throughout this paper we assume the solenoidal condition that $\nabla \cdot a = 0$.
Uniqueness, smoothness and global existence on time of solutions are very important question for this problem (for those questions, see for example \cite{Gg-Mykw,Gg-Mykw-Osd,FrwgKznShr08,FjtKt,Kzn89JDE,KznOgwTnuch02,Wisslr} and references therein).
Now we treat a smooth and global solution $u$ which satisfies that
\begin{equation}\label{assmp}
	\left\| u(t) \right\|_{L^q (\mathbb{R}^n)}
	\le
	C (1+t)^{-\frac{n}2 (1-\frac1q)-\frac12}
\end{equation}
for $1 \le q \le \infty$.
This estimate is confirmed under several frameworks (cf. \cite{Brndls,Kt,Lry,Mykw98FE,Schnbk86,Schnbk91,Wgnr87,Wgnr94}), and gives the upper bound of the decay-rate of the solution.
The lower bound of the decay-rate as $t \to +\infty$ is provided by the asymptotic expansion.
For the heat equation, we see that the decay property of a solution as $|x| \to +\infty$ is inherited from an initial data.
Thus, for the heat equation, we can derive the asymptotic expansion with arbitrary high order if we assume the fast decay for the initial data.
Whereas for \eqref{NS}, decay of $u$ as $|x| \to +\infty$ is not controlled by $a$.
Namely, even if $a \in C_0 (\mathbb{R}^n)$, then
\begin{equation}\label{spdecay}
	u(t,x) = O (|x|^{-n-1})
\end{equation}
as $|x| \to + \infty$ for any fixed $t > 0$ (cf. \cite{Brndls-Vgnrn}).
Moreover, pointwise decay of the solution is studied by many authors (see for example \cite{AmrchGrltSchnbk,Mykw02FE}).
When we try to introduce an asymptotic expansion with high-order of the similar form as in the preceding works, it is necessary that $u$ decays as $|x| \to +\infty$ sufficiently fast.
Hence the polynomial decay \eqref{spdecay} is cumbersome.
Furthermore, we are interested to far field asymptotics of the solution.
The similar problem is appearing in several dissipative equations with anomalous diffusion.
The lower bound of the decay-rate as $|x| \to +\infty$ of solutions to a semi-linear anomalous diffusion equation is studied (see \cite{Brndls-Krch,Y}).
To solve behavior as $|x| \to +\infty$ for the velocity, we employ the vorticity tensor.
The vorticity tensor $\omega_{ij} = \partial_i u_j - \partial_j u_i$ for $1 \le i,j \le n$ fulfills that
\begin{equation}\label{v}
	\partial_t \omega_{ij} - \Delta \omega_{ij} + \sum_{h = 1}^n \left( \partial_i (\omega_{hj} u_h) - \partial_j (\omega_{hi} u_h) \right) = 0,
\end{equation}
where $u_i$ is the $i$-th component of the velocity.
Moreover $\omega$ gives the velocity through the Biot-Savard law:
\begin{equation}\label{BS}
	u_j = - \sum_{k=1}^n \partial_k (-\Delta)^{-1} \omega_{kj}.
\end{equation}
Indeed, since $\nabla \cdot u = 0$, we see $\sum_{k=1}^n \partial_k \omega_{kj} = \Delta u_j - \partial_j \nabla \cdot u = \Delta u_j$.
We emphasize that decay of $\omega$ as $|x| \to +\infty$ is controlled by an initial vorticity.
Therefore an asymptotic expansion of $\omega$ as $|x| \to +\infty$ with arbitrary high-order can be defined.
This fact together with the Biot-Savard law derive an asymptotic expansion for $u$ with high-order.
Those idea firstly are established by Kukavica and Reis \cite{Kkvc-Ris}, and they showed the following estimate:
For $2 \le q \le \infty,~ m \ge 2$ and $0 \le \mu < m+n(1 - \frac1q)$,
%\begin{equation}\label{asympKR}
\[
\begin{split}
	&\biggl\| |x|^\mu \biggl(
		u_j (t)
		+
		\sum_{2 \le |\alpha| \le m} \sum_{k = 1}^n
			\frac{\nabla^\alpha \partial_k (-\Delta)^{-1} G (t)}{\alpha!} \int_{\mathbb{R}^n} (-y)^\alpha \omega_{kj} (t,y) dy
	\biggr) \biggr\|_{L^q (\mathbb{R}^n)}\\
	&=
	O \left( t^{-\frac{n}2 (1-\frac1q) - \frac12 + \frac\mu2} \right)
\end{split}
\]
%\end{equation}
%
as $t \to +\infty$.
This estimate gives the asymptotic expansion of $u$ as $|x| \to +\infty$ with arbitrary high-order.
However behavior of the coefficients $\int_{\mathbb{R}^n} (-y)^\alpha \omega (t,x) dx$ as $t \to +\infty$ is not clear, and the lower bound of the decay rate as $t \to +\infty$ is not derived.
Our goal is to clarify asymptotic profiles of $u$ as $t \to +\infty$.
Furthermore we derive the decay-rate of the solution respect to both the space and the time variables.
From \eqref{v} and \eqref{BS}, the vorticity satisfies that
\begin{equation}\label{MSomega}
\begin{split}
	\omega_{ij} (t)
	=&
	G(t) * \omega_{0ij}
	+
	\sum_{h = 1}^n \int_0^t
		\partial_j G(t-s) * \left( \omega_{hi} u_h \right) (s)
	ds
	-
	\sum_{h = 1}^n \int_0^t
		\partial_i G(t-s) * \left( \omega_{hj} u_h \right) (s)
	ds,
\end{split}
\end{equation}
where $\omega_{0ij} = \partial_i a_j - \partial_j a_i$.
The top term of the nonlinear term as $t \to +\infty$ is vanishing since
\[
	\sum_{h=1}^n \int_{\mathbb{R}^n} (\omega_{hi} u_h) (s,y) dy
	=
	\sum_{h=1}^n \int_{\mathbb{R}^n}
		\left( (\partial_h u_i - \partial_i u_h) u_h \right) (s,y)
	dy
	=
	- \int_{\mathbb{R}^n} \left( u_i \nabla \cdot u + \tfrac12 \partial_i (|u|^2) \right) dy
	= 0.
\]
Applying the Biot-Savard law to \eqref{MSomega}, we see that
\begin{equation}\label{MSu}
\begin{split}
	u_j (t)
	=&
	G(t) * a_j
	-
	\sum_{k,h=1}^n \int_0^t
		R_k R_j G (t-s) * \left( \omega_{hk} u_h \right) (s)
	ds
	-
	\sum_{h = 1}^n \int_0^t
		G(t-s) * \left( \omega_{hj} u_h \right) (s)
	ds,
\end{split}
\end{equation}
where $R_k = \partial_k (-\Delta)^{-1/2}$ is the Riesz transform.
The term of the initial velocity is represented by $G(t) * a_j = - \sum_{k=1}^n R_k (-\Delta)^{-1/2} G(t) * \omega_{0kj}$.
The velocity often is given by
\begin{equation}\label{MSuM}
	u_j (t) = G(t) * a_j - \sum_{k,h=1}^n \int_0^t \partial_h G(t-s) * P_{jk} (u_h u_k) (s) ds,
\end{equation}
where $P_{jk}$ is the Helmholtz-Fujita-Kato projection.
Before considering the behavior of the solution, we confirm that this equation and \eqref{MSu} are equivalent.
Indeed, from the solenoidal condition, the nonlinear term of \eqref{MSuM} is converted to the following:
\[
\begin{split}
	&\sum_{k,h=1}^n \int_0^t \partial_h G(t-s) * P_{jk} (u_h u_k) (s) ds\\
%	=
%	\sum_{k,h=1}^n \int_0^t G(t-s) * \left( (\delta_{jk} + R_k R_j) (u_h \partial_h u_k) \right) ds\\
	&=
	\sum_{k,h=1}^n \int_0^t R_k R_j G(t-s)*(u_h \partial_h u_k) ds
	+
	\sum_{h=1}^n \int_0^t G(t-s) * (u_h \partial_h u_j) ds\\
	&=
	\sum_{k,h=1}^n \int_0^t R_k R_j G(t-s)*(u_h \omega_{hk}) ds
	+
	\sum_{h=1}^n \int_0^t G(t-s) * (u_h \omega_{hj}) ds\\
	&+\frac12 \sum_{k=1}^n \int_0^t R_k R_j \partial_k G(t-s) * (|u|^2) (s) ds
	+\frac12 \int_0^t \partial_j G(t-s) * (|u|^2) (s) ds.
\end{split}
\]
The last two terms are canceled since
\[
	\sum_{k=1}^n R_k R_j \partial_k \varphi = \sum_{k=1}^n \mathcal{F}^{-1} \left[ \frac{i\xi_k}{|\xi|} \frac{i\xi_j}{|\xi|} i \xi_k \hat{\varphi} \right]
	= - \mathcal{F}^{-1} \left[ i\xi_j \hat{\varphi} \right] = -\partial_j \varphi
\]
for any suitable function $\varphi$.
Throughout this paper we denote the velocity by \eqref{MSu}.
The asymptotic expansion of $u$ as $t\to+\infty$ with lower-order is given by
\[
\begin{split}
	U_{j;m} (t)
	=&
	- \sum_{|\alpha| = m+1} \sum_{k=1}^n
		\frac{\nabla^\alpha R_k (-\Delta)^{-1/2} G(t)}{\alpha!}
		\int_{\mathbb{R}^n} (-y)^\alpha \omega_{0kj} (y) dy\\
	&-
	\sum_{2l + |\beta| = m} \sum_{h=1}^n \frac{\partial_t^l \nabla^\beta G (t)}{l! \beta!}
	\int_0^\infty \int_{\mathbb{R}^n}
		(-s)^l (-y)^\beta \left( \omega_{hj} u_h \right) (s,y)
	dyds
\end{split}
\]
and
\[
\begin{split}
	U_{j;m}^T (t)
	=&
	-
	\sum_{2l + |\beta| = m} \sum_{k,h=1}^n \frac{\partial_t^l \nabla^\beta R_k R_j G (t)}{l! \beta!}
	\int_0^\infty\int_{\mathbb{R}^n}
	(-s)^l (-y)^\beta \left( \omega_{hk} u_h \right) (s,y)
	dyds
\end{split}
\]
for $1 \le m \le n$.
Namely, $U_m = (U_{j;m})_{j=1}^n$ and $U_m^T = (U_{j;m}^T)_{j=1}^n$ imply that
\[
	\left\| U_{m} (t) \right\|_{L^q (\mathbb{R}^n)}
	=
	t^{-\frac{n}2 (1-\frac1q) - \frac{m}2} \left\| U_{m} (1) \right\|_{L^q (\mathbb{R}^n)}
\]
and
\[
	\left\| U_{m}^T (t) \right\|_{L^q (\mathbb{R}^n)}
	=
	t^{-\frac{n}2 (1-\frac1q) - \frac{m}2} \left\| U_{m}^T (1) \right\|_{L^q (\mathbb{R}^n)}
\]
for $1 \le q \le \infty$ and $t > 0$.
Furthermore the following estimate holds.
\begin{proposition}\label{C-FM}
	Let $n \ge 2,~ \omega_0 \in L^1 (\mathbb{R}^n) \cap L^\infty (\mathbb{R}^n)$ and $|x|^{n+1} \omega_0 \in L^1 (\mathbb{R}^n)$, and a solution $u$ of \eqref{NS} with $a_j = - \sum_{k=1}^n R_k (-\Delta)^{-1/2} \omega_{0kj}$ satisfy \eqref{assmp}.
	Then
	\[
		\Bigl\| u(t) - \sum_{k=1}^n \left( U_k + U_k^T \right) (t) \Bigr\|_{L^q (\mathbb{R}^n)}
		=
		o \bigl( t^{-\frac{n}2 (1-\frac1q) - \frac{n}2} \bigr)
	\]
	as $t \to +\infty$ holds for $1 \le q \le \infty$.
	In addition, if $|x|^{n+2} \omega_0 \in L^1 (\mathbb{R}^n)$, then
	\[
		\Bigl\| u(t) - \sum_{k=1}^n \left( U_k + U_k^T \right) (t) \Bigr\|_{L^q (\mathbb{R}^n)}
		=
		O \bigl( t^{-\frac{n}2 (1-\frac1q) - \frac{n}2 - \frac12} \log (2+t)  \bigr)
	\]
	as $t \to +\infty$.
\end{proposition}
Proposition \ref{C-FM} is a representation of the assertion in the preceding works via Carpio\cite{Crpo}, and Fujigaki and Miyakawa\cite{Fjgk-Mykw}.
In two dimensional case, the similar estimate is provided without the moment condition on the initial data, and the coefficients on the expansion are clarified (see \cite{Okb18}).
Moreover, in this preceding work, the similar estimates on the Hardy space draw a spatial decay of the solution.
Here we choose the other way to lead the spatial decay, i.e., we study the estimate with the polynomial weight.
To this we introduce the following functions for $1 \le m \le n$:
\[
	U_{j;m}^S (t)
	=
	-\sum_{2l+|\beta|=m} \sum_{k,h=1}^n \frac{\partial_t^l \nabla^\beta R_k R_j G(t)}{l!\beta!} \int_0^t \int_{\mathbb{R}^n}
		(-s)^l (-y)^\beta \bigl( \omega_{hk} u_h \bigr) (s,y)
	dyds.
\]
Behavior of the coefficients $\int_0^t \int_{\mathbb{R}^n} (-s)^l (-y)^\beta \bigl( \omega_{hk} u_h \bigr) (s,y) dyds$ as $t \to +\infty$ are not clear.
However, for $U_m^S = (U_{j;m}^S)_{j=1}^n$, we see that
\begin{equation}\label{usa0}
	\bigl\| U_{m}^T (t) - U_{m}^S (t) \bigr\|_{L^q (\mathbb{R}^n)}
	\le
	C t^{-\frac{n}2 (1-\frac1q) - \frac{n}2} (1+t)^{ - \frac12}
\end{equation}
for $1 \le q \le \infty$ (see the sentences under the proof of Proposition \ref{prop-decay-wtomg} in Section 2).
Furthermore this function fulfills the following weighted estimate.
\begin{proposition}\label{wasymp-l}
	Let $n \ge 2,~ \omega_0 \in L^1 (\mathbb{R}^n) \cap L^\infty (\mathbb{R}^n)$ and $|x|^{n+2} \omega_0 \in L^1 (\mathbb{R}^n)$, and a solution $u$ of \eqref{NS} with $a_j = - \sum_{k=1}^n R_k (-\Delta)^{-1/2} \omega_{0kj}$ satisfy \eqref{assmp}.
	Then
	\begin{equation}\label{nz}
		\left\| |x|^\mu \Bigl( u(t) - \sum_{k=1}^n \left( U_k + U_k^S \right) (t) \Bigr) \right\|_{L^q (\mathbb{R}^n)}
		=
		O \bigl( t^{-\frac{n}2 (1-\frac1q) - \frac{n}2 - \frac12 + \frac\mu{2}} \log (2+t)  \bigr)
	\end{equation}
	as $t \to +\infty$ holds for $q=1$ and $0 \le \mu \le n-1$, and for $1 < q \le \infty$ and $0 \le \mu \le n$.
\end{proposition}
Proposition \ref{wasymp-l} provides behavior of the velocity as $|x| \to + \infty$.
Indeed $\| |x|^\mu (U_k+U_k^S) \|_{L^q (\mathbb{R}^n)}$ may diverge to infinity for large $\mu$ and small $q$ (for the details of this argument, see \cite{Brndls-Krch}).
The assertion \eqref{nz} with $\mu < n(1-\frac1q)+1$ also provides the spatial profile of the velocity (cf.\cite{ChJn}).
However, to describe the far field asymptotics, we should choose $\mu$ and $q$ such that $\mu \ge n(1-\frac1q)+1$.
Propositions \ref{C-FM} and \ref{wasymp-l} give the asymptotic expansion with $n$-th order.
The renormalization yields one with higher-order, and this method requires asymptotic behavior of $\omega$.
Here we give the asymptotic expansion of $\omega$ of the Escobedo-Zuazua \cite{EZ} type.
For $2 \le m \le n+1$, let
\begin{equation}\label{defOmg}
\begin{split}
	\Omega_{hk;m} (t)
	=&
	\sum_{|\alpha|=m} \frac{\nabla^\alpha G(t)}{\alpha!} \int_{\mathbb{R}^n} (-y)^\alpha \omega_{0hk} (y) dy\\
	&+
	\sum_{2l + |\beta|=m-1} \sum_{j=1}^n \frac{\partial_t^l \nabla^\beta \partial_k G (t)}{l!\beta!} \int_0^\infty \int_{\mathbb{R}^n}
		(-s)^l (-y)^\beta \left( \omega_{jh} u_j \right) (s,y)
	dy ds\\
	&-
	\sum_{2l + |\beta| = m-1} \sum_{j=1}^n \frac{\partial_t^l \nabla^\beta \partial_h G (t)}{l!\beta!} \int_0^\infty \int_{\mathbb{R}^n}
		(-s)^l (-y)^\beta \left( \omega_{jk} u_j \right) (s,y)
	dyds
\end{split}
\end{equation}
and  $\Omega_m = (\Omega_{hk;m})_{h,k=1}^n$, then we have that $\| \Omega_{m} (t) \|_{L^q (\mathbb{R}^n)} = t^{-\frac{n}2 (1-\frac1q) - \frac{m}2} \| \Omega_{m} (1) \|_{L^q (\mathbb{R}^n)}$ for $1 \le q \le \infty$ and $t > 0$.
We define the tensor $\mathcal{I}_{n+p} = (\mathcal{I}_{hk;n+p} (t,x))_{h,k=1}^n$ for $1 \le p \le n+2$ by
\[
	\mathcal{I}_{hk;n+p} (t)
	=
	\left\{
	\begin{array}{lr}
	\displaystyle
	\sum_{i=1}^{p-2} \left( \Omega_{hk;p-i} \left( U_{h;i} + U_{h;i}^T \right) \right) (t),
	&
	3 \le p \le n+2,\\[15pt]
	0,
	&
	1 \le p \le 2.
	\end{array}
	\right.
\]
Then $\lambda^{2n+p} \mathcal{I}_{n+p} (\lambda^2 t, \lambda x) = \mathcal{I}_{n+p} (t,x)$ for $\lambda > 0$, and thus
\begin{equation}\label{decay-I}
	\bigl\| |x|^\mu \mathcal{I}_{n+p} (t) \bigr\|_{L^q (\mathbb{R}^n)}
	=
	t^{-\frac{n}2 (1-\frac1q) - \frac{n}2 - \frac{p}2 + \frac\mu{2}} \bigl\| |x|^\mu \mathcal{I}_{n+p} (1) \bigr\|_{L^q (\mathbb{R}^n)}
\end{equation}
for $1 \le q \le \infty$ and $\mu \ge 0$, and
\begin{equation}\label{int-I}
	\int_{\mathbb{R}^n}
		(-x)^\beta \mathcal{I}_{n+p} (t,x)
	dx
	=
	t^{-\frac{n}2-\frac{p}2+\frac{|\beta|}2} \int_{\mathbb{R}^n}
		(-x)^\beta \mathcal{I}_{n+p} (1,x)
	dx
\end{equation}
for $\beta \in \mathbb{Z}_+^n$.
The functions $\mathcal{I}_{hk;n+p}$ build an approximation of $\omega_{hk} u_h$ (see \eqref{bilin-asymp} and \eqref{bilin-asymp-wt}).
By using $\mathcal{I}_{n+p}$, we introduce some functions for $1 \le m \le n$:
\[
\begin{split}
	U_{j;n+m} (t)
	=&
	-\sum_{|\alpha| = n+m+1} \sum_{k=1}^n \frac{\nabla^\alpha R_k (-\Delta)^{-1/2} G(t)}{\alpha!} \int_{\mathbb{R}^n}
		(-y)^\alpha \omega_{0kj} (y)
	dy\\
	&-
	\sum_{2l+|\beta| = n+m} \frac{\partial_t^l \nabla^\beta G(t)}{l!\beta!} \sum_{h=1}^n \int_0^\infty \int_{\mathbb{R}^n}
		(-s)^l (-y)^\beta \bigl( \omega_{hj} u_h (s,y) - \sum_{p=3}^{m+1} \mathcal{I}_{hj;n+p} (s,y)\\
		&\hspace{15mm} - \mathcal{I}_{hj;n+m+2} (1+s,y) \bigr)
	dyds,\\
%\end{split}
%\]
%\[
%\begin{split}
%%	%
	U_{j;n+m}^T (t)
	=&
	-\sum_{2l+|\beta| = n+m} \sum_{k,h=1}^n \frac{\partial_t^l \nabla^\beta R_k R_j G(t)}{l!\beta!} \int_0^\infty \int_{\mathbb{R}^n}
		(-s)^l (-y)^\beta \bigl( \omega_{hk} u_h (s,y) - \sum_{p=3}^{m+1} \mathcal{I}_{hk;n+p} (s,y)\\
		&\hspace{15mm} - \mathcal{I}_{hk;n+m+2} (1+s,y) \bigr)
	dyds,
\end{split}
\]
\[
\begin{split}
%	%
	K_{j;n+m} (t)
	=&
	-\sum_{2l+|\beta|=n+m} \sum_{k,h=1}^n \frac{\partial_t^l \nabla^\beta R_k R_j G(t)}{l!\beta!} \int_0^t s^l (1+s)^{-\frac{n}2-\frac{m}2-1 + \frac{|\beta|}2} ds
	\int_{\mathbb{R}^n} (-1)^l (-y)^\beta \mathcal{I}_{hk;n+m+2} (1,y) dy\\
	&-
	\sum_{2l+|\beta|=n+m} \frac{\partial_t^l \nabla^\beta G(t)}{l!\beta!} \int_0^t s^l (1+s)^{-\frac{n}2-\frac{m}2-1 + \frac{|\beta|}2} ds
	\sum_{h=1}^n \int_{\mathbb{R}^n} (-1)^l (-y)^\beta \mathcal{I}_{hj;n+m+2} (1,y) dy,\\
%\end{split}
%\]
%\[
%\begin{split}
%	%
%	%
	V_{j;n+m} (t)
	=
	&\sum_{2l+|\beta|=1}^{n+m-1} \frac{2 t^{-\frac{n}2-\frac{m}2+l+\frac{|\beta|}2}}{n+m-2l-|\beta|}
	\frac{\partial_t^l \nabla^\beta G(t)}{l!\beta!}
	\sum_{h=1}^n \int_{\mathbb{R}^n} (-1)^l (-y)^\beta \mathcal{I}_{hj;n+m+2} (1,y) dy,\\
	V_{j;n+m}^T (t)
	=
%	- \sum_{2l+|\beta| = n + m - 1} \sum_{k,h=1}^n \frac{2 t^{-\frac12} \partial_t^l \nabla^\beta R_k R_j G(t)}{l!\beta!} \int_{\mathbb{R}^n} (-1)^l (-y)^\beta \mathcal{I}_{hk;n+m+2} (1,y) dy\\
%	&- \sum_{2l+|\beta| = n + m - 1} \frac{2 t^{-\frac12} \partial_t^l \nabla^\beta G(t)}{l!\beta!} \sum_{h=1}^n \int_{\mathbb{R}^n} (-1)^l (-y)^\beta \mathcal{I}_{hj;n+m+2} (1,y) dy\\
	&\sum_{2l+|\beta|=1}^{n+m-1} \frac{2 t^{-\frac{n}2-\frac{m}2+l+\frac{|\beta|}2}}{n+m-2l-|\beta|} \sum_{k,h=1}^n \frac{\partial_t^l \nabla^\beta R_k R_j G(t)}{l!\beta!}
	\int_{\mathbb{R}^n} (-1)^l (-y)^\beta \mathcal{I}_{hk;n+m+2} (1,y) dy,\\
%\end{split}
%\]
%\[
%\begin{split}
	\tilde{V}_{j;n+m} (t)
	=&
	-\sum_{2l+|\beta|=n+m-2} \frac{\partial_t^l \nabla^\beta G(t)}{l!\beta!}
	\int_t^\infty
		s^l \bigl( (1+s)^{-\frac{n}2-\frac{m}2+\frac{|\beta|}2} - s^{-\frac{n}2-\frac{m}2+\frac{|\beta|}2} \bigr)
	ds\\
	&\hspace{5mm}\times
	\sum_{h=1}^n \int_{\mathbb{R}^n}
		(-1)^l (-y)^\beta \mathcal{I}_{hj;n+m} (1,y)
	dy,\\
%\end{split}
%\]
%\[
%\begin{split}
%	%
	\tilde{V}_{j;n+m}^T (t)
	=&
	-\sum_{2l+|\beta|=n+m-2} \sum_{k,h=1}^n \frac{\partial_t^l \nabla^\beta R_k R_j G(t)}{l!\beta!}
	\int_t^\infty
		s^l \bigl( (1+s)^{-\frac{n}2-\frac{m}2+\frac{|\beta|}2} - s^{-\frac{n}2-\frac{m}2+\frac{|\beta|}2} \bigr)
	ds\\
	&\hspace{5mm}\times
	\int_{\mathbb{R}^n}
		(-1)^l (-y)^\beta \mathcal{I}_{hk;n+m} (1,y)
	dy,\\
%\end{split}
%\]
%\[
%\begin{split}
%	%
	J_{j;n+m} (t)
	=&
	-\sum_{k,h=1}^n \int_0^t \int_{\mathbb{R}^n} \biggl(
		R_k R_j G(t-s,x-y) - \sum_{2l+|\beta|=0}^{n+m} \frac{\partial_t^l \nabla^\beta R_k R_j G(t,x)}{l!\beta!} (-s)^l (-y)^\beta \biggr)\\
		&\hspace{5mm}\times
		\mathcal{I}_{hk;n+m+2} (s,y)
	dyds\\
	&-
	\sum_{h=1}^n \int_0^t \int_{\mathbb{R}^n} \biggl(
		G(t-s,x-y) - \sum_{2l+|\beta|=0}^{n+m} \frac{\partial_t^l \nabla^\beta G(t,x)}{l!\beta!} (-s)^l (-y)^\beta \biggr)
		\mathcal{I}_{hj;n+m+2} (s,y)
	dyds
\end{split}
\]
Here $\tilde{V}_{j;n+m} = \tilde{V}_{j;n+m}^T = 0$ for $1 \le m \le 2$ since $\mathcal{I}_{hj;n+m} = 0$.
Thus those two functions are defined only in the case $n \ge 3$.
Those functions are well-defined in $C((0,\infty);L^1 (\mathbb{R}^n) \cap L^\infty (\mathbb{R}^n))$ and satisfy that
\begin{gather}
	\left\| U_{n+m} (t) \right\|_{L^q (\mathbb{R}^n)}
	=
	t^{-\frac{n}2 (1-\frac1q) - \frac{n}2 - \frac{m}2} \left\| U_{n+m} (1) \right\|_{L^q (\mathbb{R}^n)},~
	\left\| U_{n+m}^T (t) \right\|_{L^q (\mathbb{R}^n)}
	=
	t^{-\frac{n}2 (1-\frac1q) - \frac{n}2 - \frac{m}2} \left\| U_{n+m}^T (1) \right\|_{L^q (\mathbb{R}^n)},\notag\\
	\left\| J_{n+m} (t) \right\|_{L^q (\mathbb{R}^n)}
	=
	t^{-\frac{n}2 (1-\frac1q) - \frac{n}2 - \frac{m}2} \left\| J_{n+m} (1) \right\|_{L^q (\mathbb{R}^n)},~
	\left\| V_{n+m}^T (t) \right\|_{L^q (\mathbb{R}^n)}
	=
	t^{-\frac{n}2 (1-\frac1q) - \frac{n}2 - \frac{m}2} \left\| V_{n+m}^T(1) \right\|_{L^q (\mathbb{R}^n)},\notag\\
	\left\| V_{n+m} (t) \right\|_{L^q (\mathbb{R}^n)}
	=
	t^{-\frac{n}2 (1-\frac1q) - \frac{n}2 - \frac{m}2} \left\| V_{n+m} (1) \right\|_{L^q (\mathbb{R}^n)}\label{decay-U}
\end{gather}
for $1 \le q \le \infty$ and $t > 0$, and
\begin{equation}\label{decay-tV}
	\bigl\| \tilde{V}_{n+m} (t) \bigr\|_{L^q (\mathbb{R}^n)} + \bigl\| \tilde{V}_{n+m}^T (t) \bigr\|_{L^q (\mathbb{R}^n)}
	=
	O \bigl( t^{-\frac{n}2 (1-\frac1q) - \frac{n}2 - \frac{m}2} \bigr)
\end{equation}
and
\begin{equation}\label{decay-K}
	\left\| K_{n+m} (t) \right\|_{L^q (\mathbb{R}^n)}
	=
	O \bigl( t^{-\frac{n}2 (1-\frac1q) - \frac{n}2 - \frac{m}2} \log (2+t) \bigr)
\end{equation}
as $t \to +\infty$ for $1 \le q \le \infty$.
We confirm them later (see the last sentences of Section 2).
Therefore large-time behavior of them are straightforward.
Our main assertion is established in the following theorem.
\begin{theorem}\label{thm-main}
	Let $n \ge 2,~ 1 \le m \le n,~ \omega_0 \in L^1 (\mathbb{R}^n) \cap L^\infty (\mathbb{R}^n),~ |x|^{n+m+1} \omega_0 \in L^1 (\mathbb{R}^n)$ and a solution $u$ of \eqref{NS} with $a_j = - \sum_{k=1}^n R_k (-\Delta)^{-1/2} \omega_{0kj}$ satisfy \eqref{assmp}.
	Then
\[
\begin{split}
	&\Bigl\| u (t) - \sum_{k=1}^{n+m} \left( U_{k} + U_k^T \right) (t) - \sum_{k=1}^{m} \left( K_{n+k} + V_{n+k} + V_{n+k}^T + J_{n+k} \right) (t) - \sum_{k=3}^{m} \bigl( \tilde{V}_{n+k} + \tilde{V}_{n+k}^T \bigr) (t) \Bigr\|_{L^q (\mathbb{R}^n)}\\
	&=
	o \bigl( t^{-\frac{n}2 (1-\frac1q) - \frac{n}2 - \frac{m}2} \bigr)
\end{split}
\]
as $t \to +\infty$ holds for $1 \le q \le \infty$.
\end{theorem}
This theorem yields the asymptotic expansion as $t \to +\infty$ of $2n$-th order.
The form of our expansion is complicated.
Now we emphasize that the decay-rate of any terms on the expansion with respect to both the space and the time variables is clear.
The renormalization applied as in \cite{Ishg-Kwkm-Kbys} provides the asymptotic expansion of plain form.
However the large-time behavior of the expansion obtained by this method is covered.
\begin{remark}
Upon the condition for the initial velocity that $a \in L^1 (\mathbb{R}^n) \cap L^\infty (\mathbb{R}^n)$ and $|x|^{n+m} a \in L^1 (\mathbb{R}^n)$, we also derive an asymptotic expansion of $2n$-th order.
\end{remark}
To describe far field asymptotics, we define the following functions for $1 \le m \le n$:
\[
\begin{split}
	U_{j;n+m}^S (t)
	=&
	-\sum_{2l+|\beta| = n+m} \sum_{k,h=1}^n \frac{\partial_t^l \nabla^\beta R_k R_j G(t)}{l!\beta!} \int_0^t \int_{\mathbb{R}^n}
		(-s)^l (-y)^\beta \bigl( \omega_{hk} u_h (s,y) - \sum_{p=3}^{m+1} \mathcal{I}_{hk;n+p} (s,y)\\
		&\hspace{15mm} - \mathcal{I}_{hk;n+m+2} (1+s,y) \bigr)
	dyds
\end{split}
\]
and $U_{n+m}^S = (U_{j;n+m}^S)_{j=1}^n$.
Then
\begin{equation}\label{usa1}
	\bigl\| U_{n+m}^T (t) - U_{n+m}^S (t) \bigr\|_{L^q (\mathbb{R}^n)}
	\le
	C t^{-\frac{n}2 (1-\frac1q) - \frac{n}2 - \frac{m}2} (1+t)^{ - \frac12}
\end{equation}
for $1 \le q \le \infty$ and $t > 0$.
We confirm \eqref{usa1} under the proof of Proposition \ref{prop-wtasymp-omg} in Section 2.
We establish the space-time asymptotics of the velocity with high-order in the following theorem.
\begin{theorem}\label{thm-st}
	Let $n \ge 2,~ 1 \le m \le n,~ \omega_0 \in L^1 (\mathbb{R}^n) \cap L^\infty (\mathbb{R}^n),~ |x|^{n+m+1} \omega_0 \in L^1 (\mathbb{R}^n)$ and a solution $u$ of \eqref{NS} with $a_j = - \sum_{k=1}^n R_k (-\Delta)^{-1/2} \omega_{0kj}$ satisfy \eqref{assmp}.
	Then
	\begin{equation}\label{est-st}
	\begin{split}
		&\Bigl\| |x|^\mu \Bigl( u (t) - \sum_{k=1}^{n+m} \bigl( U_{k} + U_k^S \bigr) (t) - \sum_{k=1}^m \bigl( K_{n+k} + V_{n+k} + J_{n+k} \bigr) (t) - \sum_{k=3}^m \tilde{V}_{n+k} (t) \Bigr) \Bigr\|_{L^q (\mathbb{R}^n)}\\
		&=
		o \bigl( t^{-\frac{n}2 (1-\frac1q) - \frac{n}2 - \frac{m}2 + \frac\mu{2}} \bigr)
	\end{split}
	\end{equation}
	as $t \to +\infty$ holds for $q = 1$ and $0 \le \mu \le n+m-1$, and for $1 < q \le \infty$ and $0 \le \mu \le n+m$.
\end{theorem}
\begin{remark}
Large-time behavior of the coefficient of $U_k^S$ is not straightforward but is implied by \eqref{usa0} and \eqref{usa1}.
\end{remark}
\begin{remark}
	Upon the additional condition $|x|^{n+m+2} \omega_0 \in L^1 (\mathbb{R}^n)$, the sharp estimate for \eqref{est-st} is given by $O ( t^{-\frac{n}2 (1-\frac1q) - \frac{n}2 - \frac{m}2 -\frac12 + \frac\mu{2}} L_m (t) \log (2+t) )$ as $t \to +\infty$, where
	\begin{equation}\label{L_m}
		L_m (t) = \left\{
		\begin{array}{lr}
			1,		&	1 \le m \le n-1,\\
			\log (2+t),	&	m=n.
		\end{array}
		\right.
	\end{equation}
\end{remark}
The renormalization together with Theorem \ref{thm-st} gives an asymptotic expansion with $3n$-th order.
By repeating this procedure, we can derive an asymptotic expansion with arbitrary high order.
However, large-time behavior of terms on them should be complicated.\\

\paragraph{\it Notations.}
For a vector and a tensor, we abbreviate them by using a same letter, for example, $a = (a_j)_{j=1}^n,~ b = (b_{ij})_{i,j=1}^n$.
For $x = (x_1,\ldots,x_n)$ and $y = (y_1,\ldots,y_n) \in \mathbb{R}^n$, we denote $x \cdot y = \sum_{j=1}^n x_j y_j,~ |x|^2 = x \cdot x$.
In a newline, a product of scalars is described by $\times$-symbol.
We symbolize that $\partial_t = \partial/\partial t,~ \partial_j = \partial/\partial x_j~ (1 \le j \le n),~ \nabla = (\partial_1,\ldots,\partial_n)$ and $\Delta = \sum_{j=1}^n \partial_j^2$.
The length of a multi-index $\alpha = (\alpha_1,\ldots,\alpha_n) \in \mathbb{Z}_+^n = (\mathbb{N} \cup \{ 0 \})^n$ is given by $|\alpha| = \alpha_1 + \cdots + \alpha_n$.
We abbreviate that $\alpha ! = \prod_{j=1}^n \alpha_j !,~ x^\alpha = \prod_{j=1}^n x_j^{\alpha_j}$ and $\nabla^\alpha = \prod_{j=1}^n \partial_j^{\alpha_j}$.
We define the Fourier transform and its inverse by $\hat{\varphi} (\xi) = \mathcal{F} [\varphi] (\xi) = (2\pi)^{-n/2} \int_{\mathbb{R}^n} \varphi (x) e^{-ix\cdot\xi} dx$ and $\check{\varphi} (x) = \mathcal{F}^{-1} [\varphi] (x) = (2\pi)^{-n/2} \int_{\mathbb{R}^n} \varphi (\xi) e^{ix\cdot\xi} d\xi$, respectively, where $i = \sqrt{-1}$.
For $1 \le q \le \infty$, $L^q (\mathbb{R}^n)$ denotes the Lebesgue space and $\| \cdot \|_{L^q (\mathbb{R}^n)}$ is its norm.
Various constants are simply denoted by $C$.

\section{Preliminaries}
To prove our assertions, some estimates for the vorticity are required.
\begin{proposition}\label{prop-decay-wtomg}
	Let $\omega_0 \in L^1 (\mathbb{R}^n) \cap L^\infty (\mathbb{R}^n),~ |x|^2 \omega_0 \in L^1 (\mathbb{R}^n)$ and a solution $u$ of \eqref{MSu} with $a_j = - \sum_{k=1}^n R_k (-\Delta)^{-1/2} \omega_{0kj}$ satisfy \eqref{assmp}.
	Then a solution $\omega$ of \eqref{MSomega} fulfills
	\begin{equation}\label{decay-omg}
		\| \omega (t) \|_{L^q (\mathbb{R}^n)}
		\le
		C (1+t)^{-\frac{n}2 (1-\frac1q) - 1}
	\end{equation}
	for $1 \le q \le \infty$.
	In addition, let $k \in \mathbb{Z}_+$ and $|x|^k \omega_0 \in L^1 (\mathbb{R}^n)$.
	Then
	\[
		\bigl\| |x|^k \omega (t) \bigr\|_{L^q (\mathbb{R}^n)}
		\le
		C t^{-\frac{n}2 (1-\frac1q)} (1+t)^{-1 + \frac{k}2}.
	\]
	for $1 \le q \le \infty$.
\end{proposition}
\begin{proof}
The $L^p$-$L^q$ estimate for \eqref{MSomega} together with \eqref{assmp} gives that $\| \omega (t) \|_{L^q (\mathbb{R}^n)} \le C(1+t)^{-\frac{n}2 (1-\frac1q)}$.
From $\int_{\mathbb{R}^n} \omega_{0ij} dy = 0,~ \int_{\mathbb{R}^n} y_k \omega_{0ij} dy = \int_{\mathbb{R}^n} y_k (\partial_i a_j - \partial_j a_i) dy = \int_{\mathbb{R}^n} (\delta_{kj} a_i - \delta_{ki} a_j) dy = 0$ and $\sum_{h=1}^n \int_{\mathbb{R}^n} (\omega_{hj} u_h) dy = \sum_{h=1}^n \int_{\mathbb{R}^n} (\partial_h u_j - \partial_j u_h) u_h dy = - \int_{\mathbb{R}^n} u_j \nabla \cdot u dy = 0$, \eqref{MSomega} is represented by
\begin{equation}\label{omg-bs}
\begin{split}
	\omega_{ij} (t)
	=&
	\int_{\mathbb{R}^n} \biggl( G(t,x-y) - \sum_{|\alpha| \le 1} \nabla^\alpha G(t,x) (-y)^\alpha \biggr) \omega_{0ij} (y) dy\\
	&+
	\sum_{h=1}^n \int_0^t \int_{\mathbb{R}^n}
		\left( \partial_j G(t-s,x-y) - \partial_j G(t-s,x) \right) \left( \omega_{hi} u_h \right) (s,y)
	dyds\\
	&-
	\sum_{h=1}^n \int_0^t \int_{\mathbb{R}^n}
		\left( \partial_i G(t-s,x-y) - \partial_i G(t-s,x) \right) \left( \omega_{hj} u_h \right) (s,y)
	dyds.
\end{split}
\end{equation}
From the mean value theorem, the first and the second terms are converted to
\[
	\int_{\mathbb{R}^n} \biggl( G(t,x-y) - \sum_{|\alpha| \le 1} \nabla^\alpha G(t,x) (-y)^\alpha \biggr) \omega_{0ij} (y) dy
	=
	\sum_{|\alpha| = 2} \int_{\mathbb{R}^n} \int_0^1 \frac{\nabla^\alpha G(t,x-\lambda y)}{\alpha!} \lambda (-y)^\alpha \omega_{0ij} (y) d\lambda dy
\]
and
\[
\begin{split}
	&\int_0^t \int_{\mathbb{R}^n}
		\left( \partial_j G(t-s,x-y) - \partial_j G(t-s,x) \right) \left( \omega_{hi} u_h \right) (s,y)
	dyds\\
	=&
	\int_0^{t/2} \int_{\mathbb{R}^n} \int_0^1
		\left( -y \cdot \nabla \right) \partial_j G(t-s,x-\lambda y) \left( \omega_{hi} u_h \right) (s,y)
	d\lambda dyds\\
	&+
	\int_{t/2}^t \int_{\mathbb{R}^n}
		\left( \partial_j G(t-s,x-y) - \partial_j G(t-s,x) \right) \left( \omega_{hi} u_h \right) (s,y)
	dyds,
\end{split}
\]
respectively.
The third term also is converted to the similar form.
Hence, by the Hausdorf-Young inequality and the decay of the Gauss kernel,
\begin{equation}\label{mc}
\begin{split}
	&\bigl\| \omega_{ij} (t) \bigr\|_{L^q (\mathbb{R}^n)}
	\le
	C t^{-\frac{n}2 (1-\frac1q) - 1} \bigl\| |x|^2 \omega_{0ij} \bigr\|_{L^1 (\mathbb{R}^n)}\\
	&+
	C \sum_{h=1}^n \int_0^{t/2}
		(t-s)^{-\frac{n}2 (1-\frac1q) - 1} \left( \bigl\| |x| \omega_{hi} (s) \bigr\|_{L^1 (\mathbb{R}^n)} + \bigl\| |x| \omega_{hj} (s) \bigr\|_{L^1 (\mathbb{R}^n)} \right) \bigl\| u_h (s) \bigr\|_{L^\infty (\mathbb{R}^n)}
	ds\\
	&+
	C \sum_{h=1}^n \int_{t/2}^t
		(t-s)^{-\frac12} \left( \bigl\| \omega_{hi} (s) \bigr\|_{L^q (\mathbb{R}^n)} + \bigl\| \omega_{hj} (s) \bigr\|_{L^q (\mathbb{R}^n)} \right) \bigl\| u_h (s) \bigr\|_{L^\infty (\mathbb{R}^n)}
	ds.
\end{split}
\end{equation}
For $k \ge 1$, we see from \eqref{omg-bs} that
\begin{equation}\label{24}
\begin{split}
	&|x|^k \omega_{ij} (t)
	=
	\int_{|y| \ge |x|/2} \biggl( G(t,x-y) - \sum_{|\alpha| \le 1} \nabla^\alpha G(t,x) (-y)^\alpha \biggr) |x|^k \omega_{0ij} (y) dy\\
	&+
	\sum_{|\alpha|=2} \int_{|y| \le |x|/2} \int_0^1 \frac{|x|^k \nabla^\alpha G(t,x-\lambda y)}{\alpha!} \lambda (-y)^\alpha \omega_{0ij} (y) d\lambda dy
\end{split}
\end{equation}
\[
\begin{split}
	&+
	\sum_{h=1}^n \int_0^t \int_{|y| \ge |x|/2}
		\left( \partial_j G(t-s,x-y) - \partial_j G(t-s,x) \right) |x|^k \left( \omega_{hi} u_h \right) (s,y)
	dyds\\
%\end{split}
%\end{equation}
%\[
%\begin{split}
	&+
	\sum_{|\beta|=1} \sum_{h=1}^n \int_0^t \int_{|y| \le |x|/2} \int_0^1
		|x|^k \nabla^\beta \partial_j G(t-s,x-\lambda y) (-y)^\beta  \left( \omega_{hi} u_h \right) (s,y)
	d\lambda dyds\\
%\end{split}
%\end{equation}
%\[
%\begin{split}
	&-
	\sum_{h=1}^n \int_0^t \int_{|y| \ge |x|/2}
		\left( \partial_i G(t-s,x-y) - \partial_i G(t-s,x) \right) |x|^k \left( \omega_{hj} u_h \right) (s,y)
	dyds\\
%\end{split}
%\end{equation}
%\[
%\begin{split}
	&-
	\sum_{|\beta|=1} \sum_{h=1}^n \int_0^t \int_{|y| \le |x|/2} \int_0^1
		|x|^k \nabla^\beta \partial_i G(t-s,x-\lambda y) (-y)^\beta  \left( \omega_{hj} u_h \right) (s,y)
	d\lambda dyds.
\end{split}
%\end{equation}
\]
Applying the mean value theorem to the first term with $k = 1$, we have that
\[
\begin{split}
	&\biggl\| \int_{|y| \ge |x|/2} \biggl( G(t,x-y) - \sum_{|\alpha| \le 1} \nabla^\alpha G(t,x) (-y)^\alpha \biggr) |x| \omega_{0ij} (y) dy \biggr\|_{L^q (\mathbb{R}^n)}\\
	\le&
	C \| \nabla G(t) \|_{L^q (\mathbb{R}^n)} \bigl\| |x|^2 \omega_{0ij} \bigr\|_{L^1 (\mathbb{R}^n)}
	\le
	C t^{-\frac{n}2 (1-\frac1q)-\frac12} \bigl\| |x|^2 \omega_{0ij} \bigr\|_{L^1 (\mathbb{R}^n)}.
\end{split}
\]
For $k \ge 1$, this term fulfills that
\[
\begin{split}
	&\biggl\| \int_{|y| \ge |x|/2} \biggl( G(t,x-y) - \sum_{|\alpha| \le 1} \nabla^\alpha G(t,x) (-y)^\alpha \biggr) |x|^k \omega_{0ij} (y) dy \biggr\|_{L^q (\mathbb{R}^n)}\\
	\le&
	C \left( \| G(t) \|_{L^q (\mathbb{R}^n)} + \| |x| \nabla G(t) \|_{L^q (\mathbb{R}^n)} \right) \bigl\| |x|^k \omega_{0ij} \bigr\|_{L^1 (\mathbb{R}^n)}
	\le
	C t^{-\frac{n}2 (1-\frac1q)} \bigl\| |x|^k \omega_{0ij} \bigr\|_{L^1 (\mathbb{R}^n)}.
\end{split}
\]
The second term of \eqref{24} satisfies that
\[
\begin{split}
	\biggl\| \int_{|y| \le |x|/2} \int_0^1 \frac{|x|^k \nabla^\alpha G(t,x-\lambda y)}{\alpha!} \lambda (-y)^\alpha \omega_{0ij} (y) d\lambda dy \biggr\|_{L^q (\mathbb{R}^n)}
	\le&
	C t^{-\frac{n}2 (1-\frac1q)-1+\frac{k}2} \bigl\| |x|^2 \omega_{0ij} \bigr\|_{L^1 (\mathbb{R}^n)}.
\end{split}
\]
We remark that, when $k = 1$, this norm is estimated by $C t^{-\frac{n}2 (1-\frac1q)} (1+t)^{-\frac12}$.
By using \eqref{assmp}, we see for the third and the fourth terms of \eqref{24} that
\[
\begin{split}
	&\biggl\|
		\int_0^t \int_{|y| \ge |x|/2}
			\left( \partial_j G(t-s,x-y) - \partial_j G(t-s,x) \right) |x|^k \left( \omega_{hi} u_h \right) (s,y)
		dyds
	\biggr\|_{L^q (\mathbb{R}^n)}\\
	&\le
	C \int_0^{t/2}
		(t-s)^{-\frac{n}2 (1-\frac1q)-\frac12}
		\bigl\| |x|^k \left( \omega_{hi} u_h \right) (s) \bigr\|_{L^1 (\mathbb{R}^n)}
	ds
	+
	C \int_{t/2}^t
		(t-s)^{-\frac12}
		\bigl\| |x|^k \left( \omega_{hi} u_h \right) (s) \bigr\|_{L^q (\mathbb{R}^n)}
	ds\\
	&\le
	C \int_0^{t/2}
		(t-s)^{-\frac{n}2 (1-\frac1q)-\frac12}
		(1+s)^{-\frac{n}2 - \frac12}  \bigl\| |x|^k \omega_{hi} (s) \bigr\|_{L^1 (\mathbb{R}^n)}
	ds\\
	&+
	C \int_{t/2}^t
		(t-s)^{-\frac12}
		(1+s)^{-\frac{n}2 - \frac12} \bigl\| |x|^k \omega_{hi} (s) \bigr\|_{L^q (\mathbb{R}^n)}
	ds
\end{split}
\]
and
\[
\begin{split}
	&\biggl\|
		\int_0^t \int_{|y| \le |x|/2} \int_0^1
			|x|^k \nabla^\beta \partial_j G(t-s,x-\lambda y) (-y)^\beta  \left( \omega_{hi} u_h \right) (s,y)
		d\lambda dyds
	\biggr\|_{L^q (\mathbb{R}^n)}\\
	\le&
	C \int_0^{t/2}
		(t-s)^{-\frac{n}2 (1-\frac1q) - 1 + \frac{k}2}
		\bigl\| |x| (\omega_{hi} u_h) (s) \bigr\|_{L^1 (\mathbb{R}^n)}
	ds
	+
	C \int_{t/2}^t
		(t-s)^{-1 + \frac{k}2}
		\bigl\| |x| (\omega_{hi} u_h) (s) \bigr\|_{L^q (\mathbb{R}^n)}
	ds\\
	\le&
	C \int_0^{t/2}
		(t-s)^{-\frac{n}2 (1-\frac1q) - 1 + \frac{k}2}
		(1+s)^{-\frac{n}2 - \frac12} \bigl\| |x| \omega_{hi} (s) \bigr\|_{L^1 (\mathbb{R}^n)}
	ds\\
	&+
	C \int_{t/2}^t
		(t-s)^{-1 + \frac{k}2}
		(1+s)^{-\frac{n}2-\frac12} \bigl\| |x| \omega_{hi} (s) \bigr\|_{L^q (\mathbb{R}^n)}
	ds,
\end{split}
\]
respectively.
We treat the fifth and the last terms of \eqref{24} by the similar argument, then we obtain that
\begin{equation}\label{wtomg-bs}
\begin{split}
	&\bigl\| |x|^k \omega_{ij} (t) \bigr\|_{L^q (\mathbb{R}^n)}\\
	&\le
%	C t^{-\frac{n}2 (1-\frac1q)} (1+t)^{-1+\frac{k}2} \left( \bigl\| |x|^2 \omega_{0ij} \bigr\|_{L^1 (\mathbb{R}^n)} + {\color{red} \bigl\| |x|^{k} \omega_{0ij} \bigr\|_{L^1 (\mathbb{R}^n)}}\right)\\
%	&+
%	C \sum_{|\alpha| = 2} \bigl\| |x|^k \nabla^\alpha G(t) \bigr\|_{L^q (\mathbb{R}^n)} \left\| (-x)^\alpha \omega_{0ij} \right\|_{L^1 (\mathbb{R}^n)}\\
%	&+
%	C \sum_{h=1}^n \int_0^t
%		\left\| \partial_j G(t-s) \right\|_{L^1 (\mathbb{R}^n)} \bigl\| |x|^k \omega_{hi} (s) \bigr\|_{L^q (\mathbb{R}^n)} \left\| u_h (s) \right\|_{L^\infty (\mathbb{R}^n)}
%	ds\\
%	&+
%	C \sum_{h=1}^n \sum_{|\beta|=1} \int_0^t
%		\bigl\| |x|^k \nabla^\beta \partial_j G(t-s) \bigr\|_{L^1 (\mathbb{R}^n)}
%		\bigl\| (-y)^\beta \omega_{hi} (s) \bigr\|_{L^q (\mathbb{R}^n)} \bigl\| u_h (s) \bigr\|_{L^\infty (\mathbb{R}^n)}
%	ds\\
%	&+
%	C \sum_{h=1}^n \int_0^t
%		\left\| \partial_i G(t-s) \right\|_{L^1 (\mathbb{R}^n)} \bigl\| |x|^k \omega_{hj} (s) \bigr\|_{L^q (\mathbb{R}^n)} \left\| u_h (s) \right\|_{L^\infty (\mathbb{R}^n)}
%	ds\\
%	&+
%	C \sum_{h=1}^n \sum_{|\beta|=1} \int_0^t
%		\bigl\| |x|^k \nabla^\beta \partial_i G(t-s) \bigr\|_{L^1 (\mathbb{R}^n)}
%		\bigl\| (-y)^\beta \omega_{hj} (s) \bigr\|_{L^q (\mathbb{R}^n)} \bigl\| u_h (s) \bigr\|_{L^\infty (\mathbb{R}^n)}
%	ds
%\end{split}
%\end{equation}
%\[
%\begin{split}
%	\le&
	C t^{-\frac{n}2 (1-\frac1q)} (1+t)^{-1+\frac{k}2} \left( \bigl\| |x|^2 \omega_{0ij} \bigr\|_{L^1 (\mathbb{R}^n)} + \bigl\| |x|^k \omega_{0ij} \bigr\|_{L^1 (\mathbb{R}^n)} \right)
	+
	C t^{-\frac{n}2 (1-\frac1q) -1+\frac{k}2} \left\| |x|^2 \omega_{0ij} \right\|_{L^1 (\mathbb{R}^n)}\\
%\end{split}
%\end{equation}
%\[
%\begin{split}
	&+
	C \sum_{h=1}^n \int_0^{t/2}
		(t-s)^{-\frac{n}2(1-\frac1q)-\frac12} (1+s)^{-\frac{n}2-\frac12}
		\Bigl( \bigl\| |x|^k \omega_{hi} (s) \bigr\|_{L^1 (\mathbb{R}^n)} + \bigl\| |x|^k \omega_{hj} (s) \bigr\|_{L^1 (\mathbb{R}^n)} \Bigr)
	ds\\
	&+
	C \sum_{h=1}^n \int_{t/2}^t
		(t-s)^{-\frac12} (1+s)^{-\frac{n}2-\frac12}
		\Bigl( \bigl\| |x|^k \omega_{hi} (s) \bigr\|_{L^q (\mathbb{R}^n)} + \bigl\| |x|^k \omega_{hj} (s) \bigr\|_{L^q (\mathbb{R}^n)} \Bigr)
	ds\\
%\end{split}
%\end{equation}
%\[
%\begin{split}
	&+
	C \sum_{h=1}^n \int_0^{t/2}
		(t-s)^{-\frac{n}2(1-\frac1q)-1+\frac{k}2} (1+s)^{-\frac{n}2-\frac12}
		\Bigl( \left\| |x| \omega_{hi} (s) \right\|_{L^1 (\mathbb{R}^n)} +\left\| |x| \omega_{hj} (s) \right\|_{L^1 (\mathbb{R}^n)} \Bigr)
	ds\\
	&+
	C \sum_{h=1}^n \int_{t/2}^t
		(t-s)^{-1+\frac{k}2} (1+s)^{-\frac{n}2-\frac12}
		\Bigl( \left\| |x| \omega_{hi} (s) \right\|_{L^q (\mathbb{R}^n)} +\left\| |x| \omega_{hj} (s) \right\|_{L^q (\mathbb{R}^n)} \Bigr)
	ds.
\end{split}
\end{equation}
%\]
%
When $k = q = 1$, since the singularity of the second term at $t = 0$ is removable, the Granwall estimate says that $|x|\omega(t) \in L^1 (\mathbb{R}^n)$ for $t > 0$, and
\[
\begin{split}
	&\sum_{i,j=1}^n \left\| |x| \omega_{ij} (t) \right\|_{L^1 (\mathbb{R}^n)}
	\le
	C (1+t)^{-\frac12} \sum_{i,j=1}^n \left( \left\| |x| \omega_{0ij} \right\|_{L^1 (\mathbb{R}^n)} + \left\| |x|^2 \omega_{0ij} \right\|_{L^1 (\mathbb{R}^n)} \right)\\
	&+
	C \sum_{h,i,j=1}^n \sup_{0 < \sigma < t} \left( \left\| |x| \omega_{hi} (\sigma) \right\|_{L^1 (\mathbb{R}^n)} +\left\| |x| \omega_{hj} (\sigma) \right\|_{L^1 (\mathbb{R}^n)} \right)
	\int_0^t (t-s)^{-\frac12} (1+s)^{-\frac{n}2-\frac12} ds.
\end{split}
\]
Thus we conclude that $\| |x| \omega (t) \|_{L^1 (\mathbb{R}^n)} \le C (1+t)^{-\frac12}$ and confirm \eqref{decay-omg} from \eqref{mc}.
We use this estimate into \eqref{wtomg-bs} with $k = 1$ and $1 \le q \le \infty$, then
\[
\begin{split}
	\sum_{i,j=1}^n t^{\frac{n}2 (1-\frac1q)} \bigl\| |x| \omega_{ij} (t) \bigr\|_{L^q (\mathbb{R}^n)}
	\le&
	C (1+t)^{-\frac12}
	+
	C t^{\frac{n}2 (1-\frac1q)} \int_0^{t/2} (t-s)^{-\frac{n}2 (1-\frac1q)-\frac12} (1+s)^{-\frac{n}2 - 1} ds\\
	&+
	C A_{q,1} (t)
%	\sum_{h,i,j=1}^n \sup_{0 < \sigma < t}
%	\left( \sigma^{\frac{n}2 (1-\frac1q)} \left( \bigl\| |x| \omega_{hi} (\sigma) \bigr\|_{L^q (\mathbb{R}^n)} + \bigl\| |x| \omega_{hj} (\sigma) \bigr\|_{L^q (\mathbb{R}^n)} \right) \right)
	\int_{t/2}^t (t-s)^{-\frac12} (1+s)^{-\frac{n}2 - \frac12}ds,
\end{split}
\]
where
\[
	A_{q,k} (t) = \sum_{h,i=1}^n \sup_{0 < \sigma < t} \left( \sigma^{\frac{n}2 (1-\frac1q)} \| |x|^k \omega_{hi} (\sigma) \bigr\|_{L^q (\mathbb{R}^n)} \right).
\]
Hence $\| |x| \omega (t) \|_{L^q (\mathbb{R}^n)} \le C t^{-\frac{n}2 (1-\frac1q)} (1+t)^{-\frac12}$.
Similarly \eqref{wtomg-bs} with this estimate leads that $\| |x|^k \omega (t) \|_{L^1 (\mathbb{R}^n)} \le C (1+t)^{-1+\frac{k}2}$ for $k \ge 2$.
Applying those estimates into \eqref{wtomg-bs} with $k \ge 2$ and $1 \le q \le \infty$, we see that
\[
\begin{split}
	&\sum_{i,j=1}^n t^{\frac{n}2 (1-\frac1q)} \bigl\| |x|^k \omega_{ij} (t) \bigr\|_{L^q (\mathbb{R}^n)}
	\le
	C (1+t)^{-1+\frac{k}2}
	+
	C t^{\frac{n}2 (1-\frac1q)} \int_0^{t/2} (t-s)^{-\frac{n}2 (1-\frac1q) - \frac12} (1+s)^{-\frac{n}2-\frac32 + \frac{k}2} ds\\
	&+
	C t^{\frac{n}2 (1-\frac1q)} \int_0^{t/2} (t-s)^{-\frac{n}2 (1-\frac1q) - 1 + \frac{k}2} (1+s)^{-\frac{n}2 - 1} ds
	+
	C t^{\frac{n}2 (1-\frac1q)} \int_{t/2}^t (t-s)^{-1+\frac{k}2} s^{-\frac{n}2 (1-\frac1q)} (1+s)^{-\frac{n}2 - 1} ds\\
	&+
	C A_{q,k} (t) \int_{t/2}^t (t-s)^{-\frac12} (1+s)^{-\frac{n}2 - \frac12} ds
\end{split}
\]
and then $\| |x|^k \omega (t) \|_{L^q (\mathbb{R}^n)} \le Ct^{-\frac{n}2 (1-\frac1q)} (1+t)^{-1+\frac{k}2}$.
\end{proof}
This proposition and the decay property \eqref{assmp} guarantee that $\Omega_{m+1},~ U_m,~ U_m^T,~ U_m^S$ and $\mathcal{I}_{n+m}$ for $1 \le m \le n$ are well-defined.
They also lead \eqref{usa0}.
Moreover, we see that $K_{n+m},~ V_{n+m},~ V_{n+m}^T,~ \tilde{V}_{n+m},~ \tilde{V}_{n+m}^T$ and $J_{n+m}$ employed in our main results also are well-defined.
However $J_{n+m}$ needs a special treatment (see the last sentences in this section).
We confirm the Escobedo-Zuazua type estimate for $\omega$.
\begin{proposition}\label{prop-asymp-omg}
Let $1 \le q \le \infty,~ 1 \le m \le n,~ \omega_0 \in L^1 (\mathbb{R}^n) \cap L^q (\mathbb{R}^n)$ and $|x|^{m+2} \omega_0 \in L^1 (\mathbb{R}^n)$.
Then
\[
	\biggl\| \omega (t) - \sum_{p=2}^{m+1} \Omega_{p} (t) \biggr\|_{L^q (\mathbb{R}^n)}
	\le
	C t^{-\frac{n}2 (1-\frac1q) - \frac{m}2 - \frac12} (1+t)^{-\frac12} L_m (t),
%	\quad
%	%
%	\biggl\| \omega_{ij} (t) - \sum_{p=2}^{m+3} \Omega_{ij;p}^\infty (t) \biggr\|_{L^q (\mathbb{R}^n)}
%	=
%	o \left( t^{-\frac{n}2 (1-\frac1q) - \frac{m}2 - \frac32} \right)
\]
where $\Omega_p$ and $L_m$ are defined by \eqref{defOmg} and \eqref{L_m}, respectively.
\end{proposition}
\begin{proof}
This proposition is shown by the same procedure as in \cite{EZ}.
Reader may skip this sentence.
Employing similar argument as in the proof of Proposition \ref{prop-decay-wtomg}, we see that
\begin{equation}\label{asymp-omg-base}
\begin{split}
	&\omega_{ij} (t)
	-
	\sum_{p=2}^{m+1} \Omega_{ij;p} (t)\\
	=&
	\int_{\mathbb{R}^n} \biggl( G(t,x-y) - \sum_{|\alpha| = 0}^{m+1} \frac{\nabla^\alpha G(t)}{\alpha!} (-y)^\alpha \biggr) \omega_{0ij} (y) dy\\
	&+
	\sum_{h=1}^n \int_0^t \int_{\mathbb{R}^n}
		\biggl( \partial_j G (t-s,x-y) - \sum_{2l + |\beta|=0}^m \frac{\partial_t^l \nabla^\beta \partial_j G(t,x)}{l!\beta!} (-s)^l (-y)^\beta \biggr)
		\left( \omega_{hi} u_h \right) (s,y)
	dyds\\
	&-
	\sum_{h=1}^n \int_0^t \int_{\mathbb{R}^n}
		\biggl( \partial_i G (t-s,x-y) - \sum_{2l + |\beta|=0}^m \frac{\partial_t^l \nabla^\beta \partial_i G(t,x)}{l!\beta!} (-s)^l (-y)^\beta \biggr)
		\left( \omega_{hj} u_h \right) (s,y)
	dyds.
\end{split}
\end{equation}
The estimate for the first term is straightforward.
For $N = \max \{ l \in \mathbb{Z}_+~ |~ 2l \le m \} + 1$, the second term is converted to
\[
\begin{split}
	&\int_0^t \int_{\mathbb{R}^n}
		\biggl( \partial_j G (t-s,x-y) - \sum_{2l + |\beta|=0}^m \frac{\partial_t^l \nabla^\beta \partial_j G(t,x)}{l!\beta!} (-s)^l (-y)^\beta \biggr)
		\left( \omega_{hi} u_h \right) (s,y)
	dyds\\
%\end{split}
%\]
%\[
%\begin{split}
	=&
	\int_0^{t/2} \int_{\mathbb{R}^n}
		\biggl( \partial_j G(t-s,x-y) - \sum_{2l=0}^m \frac{\partial_t^l \partial_j G(t,x-y)}{l!} (-s)^l \biggr) \left( \omega_{hi} u_h \right) (s,y)
	dyds\\
	&+
	\sum_{2l=0}^m \int_0^{t/2} \int_{\mathbb{R}^n}
		\biggl( \frac{\partial_t^l \partial_j G(t,x-y)}{l!} - \sum_{|\beta|=0}^{m-2l} \frac{\partial_t^l \nabla^\beta \partial_j G(t,x)}{l!\beta!} (-y)^\beta \biggr)
		(-s)^l \left( \omega_{hi} u_h \right) (s,y)
	dyds\\
%\end{split}
%\]
%\[
%\begin{split}
	&+
	\int_{t/2}^t \int_{\mathbb{R}^n}
		\biggl( \partial_j G (t-s,x-y) - \sum_{2l + |\beta|=1}^m \frac{\partial_t^l \nabla^\beta \partial_j G(t,x)}{l!\beta!} (-s)^l (-y)^\beta \biggr)
		\left( \omega_{hi} u_h \right) (s,y)
	dyds\\
%\end{split}
%\]
%\[
%\begin{split}
	=&
	\int_0^{t/2} \int_{\mathbb{R}^n} \int_0^1
		\frac{\partial_t^N \partial_j G(t- \lambda s,x-y)}{N!} \lambda^{N-1}
		(-s)^N \left( \omega_{hi} u_h \right) (s,y)
	d\lambda dyds\\
	&+
	\sum_{2l=0}^m \sum_{|\beta|=m+1-2l} \int_0^{t/2} \int_{\mathbb{R}^n} \int_0^1
		\frac{\partial_t^l \nabla^\beta \partial_j G(t,x-\lambda y)}{l!\beta!} \lambda^{m-2l} (-s)^l (-y)^\beta \left( \omega_{hi} u_h \right) (s,y)
	d\lambda dyds\\
%\end{split}
%\]
%\[
%\begin{split}
	&+
	\int_{t/2}^t \int_{\mathbb{R}^n}
		\biggl( \partial_j G (t-s,x-y) - \sum_{2l + |\beta|=1}^m \frac{\partial_t^l \nabla^\beta \partial_j G(t,x)}{l!\beta!} (-s)^l (-y)^\beta \biggr)
		\left( \omega_{hi} u_h \right) (s,y)
	dyds.
\end{split}
\]
Hence, by the Hausdorf-Young inequality, \eqref{assmp} and \eqref{decay-omg}, we have that
\[
\begin{split}
	&\biggl\|
		\int_0^t \int_{\mathbb{R}^n}
			\biggl( \partial_j G (t-s,x-y) - \sum_{2l + |\beta|=0}^m \frac{\partial_t^l \nabla^\beta \partial_j G(t,x)}{l!\beta!} (-s)^l (-y)^\beta \biggr)
			\left( \omega_{hi} u_h \right) (s,y)
		dyds
	\biggr\|_{L^q (\mathbb{R}^n)}\\
	\le&
	C t^{-\frac{n}2 (1-\frac1q) - N - \frac12} \int_0^{t/2}
		s^N \left\| (\omega_{hi} u_h) (s) \right\|_{L^1 (\mathbb{R}^n)}
	ds\\
	&+
	C t^{-\frac{n}2 (1-\frac1q) - \frac{m}2 - 1} \sum_{2l=0}^m \int_0^{t/2}
		s^l \bigl\| |y|^{m+1-2l} \left( \omega_{hi} u_h \right) (s) \bigr\|_{L^1 (\mathbb{R}^n)}
	ds
	+
	C \int_{t/2}^t
		(t-s)^{-\frac12} \left\| \left( \omega_{hi} u_h \right) (s) \right\|_{L^q (\mathbb{R}^n)}
	ds
\end{split}
\]
\[
\begin{split}
	&+
	C \sum_{2l+|\beta|=1}^m t^{-\frac{n}2 (1-\frac1q) - l - \frac{|\beta|}2 - \frac12} \int_{t/2}^t
		s^l \bigl\| (-y)^\beta \left( \omega_{hi} u_h \right) (s) \bigr\|_{L^1 (\mathbb{R}^n)}
	ds\\
%\end{split}
%\]
%\[
%\begin{split}
	\le&
	C t^{-\frac{n}2 (1-\frac1q) - N - \frac12} \int_0^{t/2}
		(1+s)^{-\frac{n}2 - \frac32 + N}
	ds
	+
	C t^{-\frac{n}2 (1-\frac1q) - \frac{m}2 - 1} \int_0^{t/2}
		(1+s)^{-\frac{n}2-1+\frac{m}2}
	ds\\
	&+
	C \int_{t/2}^t
		(t-s)^{-\frac12} s^{-\frac{n}2 (1-\frac1q) -\frac{m}2 - 1}
	ds
	+
	C \sum_{2l+|\beta|=1}^m t^{-\frac{n}2 (1-\frac1q) - l - \frac{|\beta|}2 - \frac12} \int_{t/2}^t
		s^{-\frac{n}2 - \frac32 + l + \frac{|\beta|}2}
	ds\\
%\end{split}
%\]
%\[
%\begin{split}
	\le&
	C t^{-\frac{n}2 (1-\frac1q) - \frac{m}2 - 1} L_m (t).
\end{split}
\]
Similar treatment provides the estimate for the last term on \eqref{asymp-omg-base}.
\end{proof}
Now we see that
\begin{equation}\label{Izr}
	\int_{\mathbb{R}^n} \mathcal{I}_{n+m} (t,x) dx = 0
\end{equation}
for $3 \le m \le n+2$ and $t > 0$.
Indeed, for $m = 3$, if we assume $\int_{\mathbb{R}^n} \mathcal{I}_{hj;n+3} (t,x) dx \neq 0$ for some $1 \le h,j \le n$, then
\[
	\int_{\mathbb{R}^n} \left( \omega_{hj} u_h - \mathcal{I}_{hj;n+3} \right) (t,x) dx
	=
	- \int_{\mathbb{R}^n} \mathcal{I}_{hj;n+3} (t,x) dx
	=
	- t^{-\frac{n}2 - \frac32} \int_{\mathbb{R}^n} \mathcal{I}_{hj;n+3} (1,x) dx.
\]
On the other hand \eqref{assmp} and Propositions \ref{C-FM}, \ref{prop-decay-wtomg} and \ref{prop-asymp-omg} say that
\[
	\biggl| \int_{\mathbb{R}^n} \left( \omega_{hj} u_h - \mathcal{I}_{hj;n+3} \right) (t,x) dx \biggr|
	\le
	\bigl\| ( \omega_{hj} u_h - \mathcal{I}_{hj;n+3} ) (t) \bigr\|_{L^1 (\mathbb{R}^n)}
	=
	o \bigl( t^{-\frac{n}2 - \frac32} \bigr)
\]
as $t \to +\infty$.
They are contradictory.
Inductively, if $\int_{\mathbb{R}^n} \mathcal{I}_{hj;n+m} (t,x) dx \neq 0$, then
\[
	\int_{\mathbb{R}^n} \biggl(
		\omega_{hk} u_h - \sum_{p=3}^{m} \mathcal{I}_{hj;n+p}
	\bigg) (t,x) dx
	=
	- \int_{\mathbb{R}^n} \mathcal{I}_{hj;n+m} (t,x) dx
	=
	- t^{-\frac{n}2 - \frac{m}2} \int_{\mathbb{R}^n} \mathcal{I}_{hj;n+m} (1,x) dx.
\]
However
\[
	\biggl| \int_{\mathbb{R}^n} \biggl(
		\omega_{hk} u_h - \sum_{p=3}^{m} \mathcal{I}_{hj;n+p}
	\bigg) (t,x) dx \biggr|
	\le
	\biggl\| \biggl(
		\omega_{hk} u_h - \sum_{p=3}^{m} \mathcal{I}_{hj;n+p}
	\bigg) (t) \biggr\|_{L^1 (\mathbb{R}^n)}
	=
	o \bigl( t^{-\frac{n}2 - \frac{m}2} \bigr)
\]
as $t \to +\infty$.
Therefore $\int_{\mathbb{R}^n} \mathcal{I}_{hj;n+m} (t,x) dx = 0$ for any $1 \le h,j \le n$.
We prepare the following weighted estimate.
\begin{proposition}\label{prop-wtasymp-omg}
Let $1 \le m \le n,~ (1+|x|)^{n+m+1} \omega_0 \in L^1 (\mathbb{R}^n),~ 1 \le q \le \infty$ and $0 \le \mu \le n+m+1$.
Then
\[
\begin{split}
	&\biggl\| |x|^\mu \biggl( \omega (t) - \sum_{p=2}^{m+1} \Omega_{p} (t) \biggr) \biggr\|_{L^q (\mathbb{R}^n)}
	\le
	C t^{-\frac{n}2(1-\frac1q)} \left( t^{ -\frac{m}2-\frac12+\frac{\mu}2} + (1+t)^{ -\frac{m}2-\frac12+\frac{\mu}2}\right) (1+t)^{-\frac12} L_m (t),
\end{split}
\]
where $\Omega_p$ and $L_m$ are defined by \eqref{defOmg} and \eqref{L_m}, respectively.
\end{proposition} 
\begin{proof}
Proposition \ref{prop-decay-wtomg} and the definition of $\Omega_p$ immediately gives
\[
	\biggl\| |x|^\mu \biggl( \omega (t) - \sum_{p=2}^{m+1} \Omega_{p} (t) \biggr) \biggr\|_{L^q (\mathbb{R}^n)}
	\le
	C t^{-\frac{n}2 (1-\frac1q)} \left( t^{-1+\frac{\mu}2} + (1+t)^{-\frac{m}2 - \frac12 + \frac{\mu}2} \right).
\]
We firstly choose $\mu = n+m+1$.
We treat the right hand side of \eqref{asymp-omg-base}.
The first term is separated to
\[
\begin{split}
	&\int_{\mathbb{R}^n} \biggl( G(t,x-y) - \sum_{|\alpha| = 0}^{m+1} \frac{\nabla^\alpha G(t)}{\alpha!} (-y)^\alpha \biggr) \omega_{0ij} (y) dy\\
	=&
	\int_{|y| \ge |x|/2} \biggl( G(t,x-y) - \sum_{|\alpha| = 0}^{m+1} \nabla^\alpha G(t,x) (-y)^\alpha \biggr) \omega_{0ij} (y) dy\\
	&+
	\sum_{|\alpha|=m+2} \int_{|y| \le |x|/2} \int_0^1 \frac{\nabla^\alpha G(t,x-\lambda y)}{\alpha!} \lambda^{m+1} (-y)^\alpha \omega_{0ij} (y) d\lambda dy.
\end{split}
\]
Then
\[
\begin{split}
	&\biggl\| |x|^{n+m+1} \int_{\mathbb{R}^n} \biggl( G(t,x-y) - \sum_{|\alpha| = 0}^{m+1} \frac{\nabla^\alpha G(t)}{\alpha!} (-y)^\alpha \biggr) \omega_{0ij} (y) dy \biggr\|_{L^q (\mathbb{R}^n)}\\
	\le&
	C \biggl(
		\bigl\| G(t) \bigr\|_{L^q (\mathbb{R}^n)} \bigl\| |x|^{n+m+1} \omega_{0ij} \bigr\|_{L^1 (\mathbb{R}^n)}
		+ \sum_{|\alpha|=0}^{m+1} \bigl\| |x|^{|\alpha|} \nabla^\alpha G(t) \bigr\|_{L^q (\mathbb{R}^n)} \bigl\| |x|^{n+m+1-|\alpha|} (-x)^\alpha \omega_{0ij} \bigr\|_{L^1 (\mathbb{R}^n)}
	\biggr)\\
	&+
	C \sum_{|\alpha|=m+2} \bigl\| |x|^{n+m+1} \nabla^\alpha G(t) \bigr\|_{L^q (\mathbb{R}^n)} \bigl\| (-x)^\alpha \omega_{0ij} \bigr\|_{L^1 (\mathbb{R}^n)}
	\le
	C t^{-\frac{n}2 (1-\frac1q)} (1+t)^{-\frac12 + \frac{n}2}.
\end{split}
\]
For the second term of \eqref{asymp-omg-base}, we split the domain $(0,t)\times\mathbb{R}^n$ to
\begin{gather}
	Q_1 = (0,t/2] \times \{ y \in \mathbb{R}^n~ |~ |y| > |x|/2 \},\quad
	Q_2 = (0,t) \times \{ y \in \mathbb{R}^n~ |~ |y| \le |x|/2 \},\label{Q}\\
	Q_3 = (t/2,t) \times \{ y \in \mathbb{R}^n~ |~ |y| > |x|/2 \},\quad
	Q_4 = Q_2,\quad
	Q_5 = Q_1 \cup Q_3.\notag
\end{gather}
Then
\[
	\int_0^t \int_{\mathbb{R}^n}
		\biggl( \partial_j G (t-s,x-y) - \sum_{2l + |\beta|=0}^m \frac{\partial_t^l \nabla^\beta \partial_j G(t,x)}{l!\beta!} (-s)^l (-y)^\beta \biggr)
		\left( \omega_{hi} u_h \right) (s,y)
	dyds
	=
	\rho_1 (t) + \cdots + \rho_5 (t),
\]
where
\[
	\rho_k (t)
	=
	\left\{
	\begin{array}{lr}
	\displaystyle
		\iint_{Q_k}
			\biggl( \partial_j G(t-s,x-y) - \sum_{2l=0}^m \frac{\partial_t^l \partial_j G(t,x-y)}{l!} (-s)^l \biggr)
			\left( \omega_{hi} u_h \right) (s,y)
		dyds,
		&
		k = 1,2,3,\\
	\displaystyle
		\sum_{2l=0}^m \iint_{Q_k}
			\biggl( \frac{\partial_t^l \partial_j G(t,x-y)}{l!} - \sum_{|\beta|=0}^{m-2l} \frac{\partial_t^l \nabla^\beta \partial_j G(t,x)}{l!\beta!} (-y)^\beta \biggr)
			(-s)^l \left( \omega_{hi} u_h \right) (s,y)
		dyds,
		&
		k = 4,5.
	\end{array}
	\right.
\]
The Taylor theorem leads that
\[
	\rho_1 (t)
	=
	\int_0^{t/2} \int_{|y| > |x|/2} \int_0^1
		\frac{\partial_t^N \partial_j G(t-\lambda s, x-y)}{N!} \lambda^{N-1} (-s)^N \left( \omega_{hi} u_h \right) (s,y)
	d\lambda dyds
\]
and
\[
	\rho_2 (t)
	=
	\int_0^t \int_{|y| \le |x|/2} \int_0^1
		\frac{\partial_t^N \partial_j G(t-\lambda s, x-y)}{N!} \lambda^{N-1} (-s)^N \left( \omega_{hi} u_h \right) (s,y)
	d\lambda dyds
\]
for $N = \max \{ l \in \mathbb{Z}_+~ |~ 2l \le m \} + 1$.
Hence, from \eqref{assmp} and Proposition \ref{prop-decay-wtomg},
\[
\begin{split}
	\bigl\| |x|^{n+m+1} \rho_1 (t) \bigr\|_{L^q (\mathbb{R}^n)}
	\le&
	C \int_0^{t/2}
		(t-s)^{-\frac{n}2 (1-\frac1q)-N-\frac12}
		s^N \left\| |x|^{n+m+1} \left( \omega_{hi} u_h \right) (s) \right\|_{L^1 (\mathbb{R}^n)}
	ds\\
	\le&
	C \int_0^{t/2}
		(t-s)^{-\frac{n}2 (1-\frac1q) -N-\frac12} s^N (1+s)^{- 1 + \frac{m}2}
	ds
	\le
	C t^{-\frac{n}2 (1-\frac1q) - \frac12 + \frac{m}2}
\end{split}
\]
and
\[
\begin{split}
	\bigl\| |x|^{n+m+1} \rho_2 (t) \bigr\|_{L^q (\mathbb{R}^n)}
	\le&
	C \int_0^{t/2}
		(t-s)^{-\frac{n}2(1-\frac1q)-N+\frac{n}2+\frac{m}2}
		s^N \left\| \left( \omega_{hi} u_h \right) (s) \right\|_{L^1 (\mathbb{R}^n)}
	ds\\
	&+
	C \int_{t/2}^t
		(t-s)^{-N+\frac{n}2+\frac{m}2}
		s^N \left\| \left( \omega_{hi} u_h \right) (s) \right\|_{L^q (\mathbb{R}^n)}
	ds\\
	\le&
	C \int_0^{t/2}
		(t-s)^{-\frac{n}2 (1-\frac1q) -N+\frac{n}2 + \frac{m}2}
		s^N (1+s)^{-\frac{n}2 - \frac32}
	ds\\
	&+
	C \int_{t/2}^t
		(t-s)^{-N+\frac{n}2 + \frac{m}2}
		s^N (1+s)^{-\frac{n}2 (1-\frac1q) -\frac{n}2 - \frac32}
	ds\\
	\le&
	C t^{-\frac{n}2 (1-\frac1q)} (1+t)^{-\frac12 + \frac{m}2} L_m (t).
\end{split}
\]
By \eqref{assmp} and Proposition \ref{prop-decay-wtomg}, we have that
\[
\begin{split}
	\bigl\| |x|^{n+m+1} \rho_3 (t) \bigr\|_{L^q (\mathbb{R}^n)}
	\le&
	C \int_{t/2}^t (t-s)^{-\frac12} \left\| |y|^{n+m+1} \left( \omega_{hi} u_h \right) (s) \right\|_{L^q (\mathbb{R}^n)} ds\\
	&+
	C \sum_{2l=0}^m t^{-\frac{n}2 (1-\frac1q) -l - \frac12} \int_{t/2}^t s^l  \left\| |y|^{n+m+1} \left( \omega_{hi} u_h \right) (s) \right\|_{L^1 (\mathbb{R}^n)} ds\\
%\end{split}
%\]
%\[
%\begin{split}
	\le&
	C \int_{t/2}^t (t-s)^{-\frac12} s^{-\frac{n}2 (1-\frac1q)} (1+s)^{-1+\frac{m}2} ds
	+
	C \sum_{2l=0}^m t^{-\frac{n}2 (1-\frac1q) -l-\frac12} \int_{t/2}^t s^l (1+s)^{-1+\frac{m}2} ds\\
	\le&
	C t^{-\frac{n}2 (1-\frac1q)} (1+t)^{- \frac12 + \frac{m}2}
	\le
	C t^{-\frac{n}2 (1-\frac1q)} (1+t)^{-\frac12 + \frac{n}2}.
\end{split}
\]
Since
\[
\begin{split}
	|x|^{n+m+1} \rho_4 (t)
	=&
	\sum_{2l=0}^m \sum_{|\beta| = m - 2l + 1} \int_0^t \int_{|y| \le |x|/2} \int_0^1
		|x|^{n+m+1} \frac{\partial_t^l \nabla^\beta \partial_j G(t,x-\lambda y)}{l!\beta!}\\
		&\hspace{5mm}\times
		\lambda^{m-2l} (-s)^l (-y)^\beta \left( \omega_{hi} u_h \right) (s,y)
	d\lambda dyds,
\end{split}
\]
we obtain that
\[
\begin{split}
	\bigl\| |x|^{n+m+1} \rho_4 (t) \bigr\|_{L^q (\mathbb{R}^n)}
	\le&
	C t^{-\frac{n}2 (1-\frac1q)-\frac12+\frac{n}2} \sum_{2l=0}^m \int_0^t s^l (1+s)^{- \frac{n}2 - 1 - l + \frac{m}2} ds\\
	\le&
	C t^{-\frac{n}2 (1-\frac1q)-\frac12+\frac{n}2} L_m (t).
\end{split}
\]
The last term fulfills that
\[
\begin{split}
	\bigl\| |x|^{n+m+1} \rho_5 (t) \bigr\|_{L^q (\mathbb{R}^n)}
	\le&
	C \sum_{2l=0}^m t^{-\frac{n}2 (1-\frac1q)-\frac12 - l}
	\int_0^t
		s^l \bigl\| |y|^{n+m+1} \left( \omega_{hi} u_h \right) (s) \bigr\|_{L^1 (\mathbb{R}^n)}
	ds\\
	\le&
	C \sum_{2l=0}^m t^{-\frac{n}2 (1-\frac1q)-\frac12 - l} \int_0^t s^l (1+s)^{- 1 + \frac{m}2} ds
	\le
	C t^{-\frac{n}2 (1-\frac1q)} (1+t)^{-\frac12 + \frac{n}2}.
\end{split}
\]
The last term of \eqref{asymp-omg-base} is treated by the similar estimates.
Therefore we get the desired estimate with $\mu = n+ m + 1$.
The coupling of this and Proposition \ref{prop-asymp-omg} completes the proof.
\end{proof}
Proposition \ref{prop-wtasymp-omg} never give far field asymptotics of $\omega$ since $\| |x|^\mu \Omega_p \|_{L^q (\mathbb{R}^n)}$ is integrable for any large $\mu$.
This proposition is prepared to prove our main assertions.
The above inequalities lead for $1 \le m \le n$ and $1 \le q \le \infty$ that
\begin{equation}\label{bilin-asymp}
\begin{split}
	\biggl\|
		\omega_{hk} u_h - \sum_{p=3}^{m+2} \mathcal{I}_{hk;n+p}
	\biggr\|_{L^q (\mathbb{R}^n)}
	\le&
	\left\| \omega_{hk} \right\|_{L^q (\mathbb{R}^n)}
	\bigl\| u_h - \sum_{i=1}^m (U_{h;i} + U_{h;i}^T) \bigr\|_{L^\infty (\mathbb{R}^n)}\\
	&+
	\sum_{i=1}^m \bigl\| \omega_{hk} - \sum_{p=2}^{m+2-i} \Omega_{hk;p} \bigr\|_{L^q (\mathbb{R}^n)}
	\bigl\| U_{h;i} + U_{h;i}^T \bigr\|_{L^\infty (\mathbb{R}^n)}\\
	\le&
	C t^{-\frac{n}2 (1 - \frac1q) - \frac{n}2 - \frac{m}2 - 1} (1+t)^{-\frac12} L_m (t).
\end{split}
\end{equation}
Upon the condition $|x|^{n+m+1} \omega_0 \in L^1 (\mathbb{R}^n)$, we have for $0 \le \mu \le n+m+1$ that
\[
\begin{split}
	&\biggl\|
		|x|^\mu \biggl( \omega_{hk} u_h - \sum_{p=3}^{m+2} \mathcal{I}_{hk;n+p} \biggr)
	\biggr\|_{L^q (\mathbb{R}^n)}
	\le
	\bigl\| |x|^\mu \omega_{hk} \bigr\|_{L^q (\mathbb{R}^n)}
	\biggl\| u_h - \sum_{i=1}^m (U_{h;i} + U_{h;i}^T) \biggr\|_{L^\infty (\mathbb{R}^n)}\\
	&+
	\sum_{i=1}^m \biggl\| |x|^{\mu} \bigl( \omega_{hk} - \sum_{p=2}^{m+2-i} \Omega_{hk;p} \bigr) \biggr\|_{L^q (\mathbb{R}^n)}
	\bigl\| U_{h;i} + U_{h;i}^T \bigr\|_{L^\infty (\mathbb{R}^n)}\\
	\le&
	C t^{- \frac{n}2 (1 - \frac1q) - \frac{n}2 - \frac{m}2} \left( t^{-1 + \frac\mu{2}} + (1+t)^{ -1 + \frac\mu{2}} \right) (1+t)^{-\frac12} L_m (t).
\end{split}
\]
We relieve the singularity at $t = 0$ by using the Minkowski inequality together with \eqref{assmp}, Proposition \ref{prop-decay-wtomg} and \eqref{int-I}, then
\begin{equation}\label{bilin-asymp-wt}
\begin{split}
	&\biggl\|
		|x|^\mu \biggl( \omega_{hk} u_h - \sum_{p=3}^{m+2} \mathcal{I}_{hk;n+p} \biggr)
	\biggr\|_{L^q (\mathbb{R}^n)}
	\le
	C t^{- \frac{n}2 (1-\frac1q) -\frac{n}2 - \frac{m}2 - 1 + \frac\mu{2}} (1+t)^{-\frac12} L_m (t).
\end{split}
\end{equation}
Since
\[
	\omega_{hk} u_h (t) - \sum_{p=3}^{m+1} \mathcal{I}_{hk;n+p} (t) - \mathcal{I}_{hk;n+m+2} (1+t)
	=
	\omega_{hk} u_h (t) - \sum_{p=3}^{m+2} \mathcal{I}_{hk;n+p} (t) - \int_0^1 \partial_t \mathcal{I}_{hk;n+m+2} (t + \lambda) d\lambda,
\]
$\partial_t \mathcal{I}_{hk;n+p} (t,x) = t^{-n-\frac{p}2 -1} \partial_t \mathcal{I}_{hk;n+p} (1,t^{-\frac12}x)$, and $\partial_t \mathcal{I}_{hk;n+p} (1,x) \in L^q (\mathbb{R}^n)$ for $1 \le q \le \infty$, we see that
\[
\begin{split}
	&\biggl\|
		|x|^\mu \biggl( \omega_{hk} u_h (t) - \sum_{p=3}^{m+1} \mathcal{I}_{hk;n+p} (t) - \mathcal{I}_{hk;n+m+2} (1+t) \biggr)
	\biggr\|_{L^q (\mathbb{R}^n)}\\
	&\le
	C t^{- \frac{n}2 (1 - \frac1q)} \left( t^{- \frac{n}2 - \frac{m}2 - \frac12 + \frac\mu{2}} + (1+t)^{- \frac{n}2 - \frac{m}2 - \frac12 + \frac\mu{2}} \right) (1+t)^{-1} L_m (t).
\end{split}
\]
Those inequalities play important role in the proof of our assertions.
Moreover they guarantee that $U_{n+m},~ U_{n+m}^T$ and $U_{n+m}^S$ for $1 \le m \le n$ are well-defined in $L^1 (\mathbb{R}^n) \cap L^\infty (\mathbb{R}^n)$, and \eqref{usa1} holds.
We show that $J_{n+m}$ is well-defined.
Indeed, by the similar calculus as in the proof of Proposition \ref{prop-asymp-omg}, the first term of $J_{n+m}$ is represented by
\[
\begin{split}
	&\int_0^t \int_{\mathbb{R}^n}
	\biggl(
		R_k R_j G(t-s,x-y) - \sum_{2l+|\beta|=0}^{n+m} \frac{\partial_t^l \nabla^\beta R_k R_j G(t,x)}{l!\beta!} (-s)^l (-y)^\beta
	\biggr)
	\mathcal{I}_{hk;n+m+2} (s,y)
	dyds\\
	&=
	\int_0^{t/2} \int_{\mathbb{R}^n} \int_0^1
		\frac{\partial_t^N R_k R_j G(t-\lambda s,x-y)}{N!}
		\lambda^{N-1} (-s)^N \mathcal{I}_{hk;n+m+2} (s,y)
	d\lambda dyds\\
	&+
	\sum_{2l=0}^m \sum_{|\beta| = n+m+1-2l} \int_0^{t/2} \int_{\mathbb{R}^n} \int_0^1
		\frac{\partial_t^l \nabla^\beta R_k R_j G(t,x-\lambda y)}{l!\beta!} \lambda^{n+m-2l}
		(-s)^l (-y)^\beta \mathcal{I}_{hk;n+m+2} (s,y)
	d\lambda dyds\\
	&+
	\int_{t/2}^t \int_{\mathbb{R}^n} \int_0^1
		\nabla R_k R_j G(t-s,x-\lambda y)
		\cdot (-y) \mathcal{I}_{hk;n+m+2} (s,y)
	d\lambda dyds\\
	&-
	\sum_{2l+|\beta|=1}^{n+m} \frac{\partial_t^l \nabla^\beta R_k R_j G(t,x)}{l!\beta!} \int_{t/2}^t \int_{\mathbb{R}^n}
	(-s)^l (-y)^\beta
	\mathcal{I}_{hk;n+m+2} (s,y)
	dyds,
\end{split}
\]
where $N = \max \{ l \in \mathbb{Z}~ |~ 2l \le n+m \} + 1$.
Here we used \eqref{Izr}.
Hence, by \eqref{decay-I}, we see for $1 \le q \le \infty$ that
\[
\begin{split}
	&\biggl\| \int_0^t \int_{\mathbb{R}^n}
	\biggl(
		R_k R_j G(t-s,x-y) - \sum_{2l+|\beta|=0}^{n+m} \frac{\partial_t^l \nabla^\beta R_k R_j G(t,x)}{l!\beta!} (-s)^l (-y)^\beta
	\biggr)
	\mathcal{I}_{hk;n+m+2} (s,y)
	dyds \biggr\|_{L^q (\mathbb{R}^n)}\\
	&\le
	C \int_0^{t/2}
		(t-s)^{-\frac{n}2 (1-\frac1q)-N}
		s^{-\frac{n}2-\frac{m}2-1+N}
	ds
	+
	C t^{-\frac{n}2 (1-\frac1q)-\frac{n}2-\frac{m}2-\frac12} \int_0^{t/2}
		s^{-\frac12}
	ds\\
	&+
	C \int_{t/2}^t
		(t-s)^{-\frac12}
		s^{-\frac{n}2 (1-\frac1q)-\frac{n}2-\frac{m}2-\frac12}
	ds
	+
	C \sum_{2l+|\beta|=1}^{n+m} t^{-\frac{n}2 (1-\frac1q)-l-\frac{|\beta|}2} \int_{t/2}^t
		s^{-\frac{n}2-\frac{m}2-1+l+\frac{|\beta|}2}
	ds.
\end{split}
\]
The right hand side is integrable for any fixed $t > 0$.
The second term of $J_{n+m}$ is treated in the similar way.
Therefore $J_{n+m}$ also is well-defined in $L^1 (\mathbb{R}^n) \cap L^\infty (\mathbb{R}^n)$.
The decay properties \eqref{decay-U} are coming from the scaling property of those functions.
The estimates \eqref{decay-tV} and \eqref{decay-K} are straightforward.
\section{Proof of main results}
%%
%Applying to \eqref{MSu} the similar calculus as in the proof of Proposition \ref{prop-wtasymp-omg}, we obtain Proposition \ref{wasymp-l}.
%Here we omit the proof.
In this section, we firstly prove Theorem \ref{thm-st}.
This proof also show Proposition \ref{prop-wtasymp-omg} (\eqref{s0}, and \eqref{est-wtr} with $m = 0$ immediately gives this proposition).\\

\paragraph{\it Proof of Theorem \ref{thm-st}.}
Firstly, we derive the asymptotic expansion.
Since $\int_{\mathbb{R}^n} \omega_{hk} u_h dx = 0$, $u$ denoted by \eqref{MSu} is expanded to
\begin{equation}\label{s0}
\begin{split}
	u_j (t)
	=&
	\sum_{m=1}^n (U_{j;m} + U_{j;m}^S) (t)
%	- \sum_{|\alpha| = 2}^{n+1} \sum_{k=1}^n \frac{\nabla^\alpha R_k (-\Delta)^{-1/2} G (t)}{\alpha!} \int_{\mathbb{R}^n} (-y)^\alpha \omega_{0kj} (y) dy\\
%	&+
%	\sum_{2l + |\beta| = 1}^n \sum_{k,h=1}^n \frac{\partial_t^l \nabla^\beta R_k R_j G (t)}{l! \beta!}
%	\int_0^t \int_{\mathbb{R}^n}
%		(-s)^l (-y)^\beta \left( \omega_{hk} u_h \right) (s,y)
%	dyds\\
%	&+
%	\sum_{2l + |\beta| = 1}^n \sum_{h=1}^n \frac{\partial_t^l \nabla^\beta G (t)}{l! \beta!}
%	\int_0^\infty \int_{\mathbb{R}^n}
%		(-s)^l (-y)^\beta \left( \omega_{hj} u_h \right) (s,y)
%	dyds\\
	+ r_{0,n} (t) + r_{1,n} (t) + r_{2,n} (t) + r_{3,n} (t),
\end{split}
\end{equation}
where
\[
\begin{split}
	r_{0,n} (t)
	=&
	- \sum_{k=1}^n R_k (-\Delta)^{-1/2} G(t) * \omega_{0kj}
	+ \sum_{|\alpha| = 0}^{n+1} \sum_{k=1}^n \frac{\nabla^\alpha R_k (-\Delta)^{-1/2} G(t)}{\alpha!} \int_{\mathbb{R}^n} (-y)^\alpha \omega_{0kj} (y) dy,\\
%\end{split}
%\]
%\[
%\begin{split}
%	%
	r_{1,n} (t)
	=&
	- \sum_{k,h=1}^n \int_0^t \int_{\mathbb{R}^n}
		\biggl(
			R_k R_j G(t-s,x-y) - \sum_{2l + |\beta| = 0}^n \frac{\partial_t^l \nabla^\beta R_k R_j G (t,x)}{l!\beta!} (-s)^l (-y)^\beta
		\biggr)\\
		&\hspace{15mm}\times
		\left( \omega_{hk} u_h \right) (s,y)
	dyds,\\
%\end{split}
%\]
%\[
%\begin{split}
%	%
	r_{2,n} (t)
	=&
	- \sum_{h=1}^n \int_0^t \int_{\mathbb{R}^n}
		\biggl(
			G(t-s,x-y) - \sum_{2l + |\beta| = 0}^n \frac{\partial_t^l \nabla^\beta G (t,x)}{l!\beta!} (-s)^l (-y)^\beta
		\biggr)
		\left( \omega_{hj} u_h \right) (s,y)
	dyds,\\
	r_{3,n} (t)
	=&
	\sum_{2l+|\beta|=1}^n \frac{\partial_t^l \nabla^\beta G(t)}{l!\beta!} \sum_{h=1}^n \int_t^\infty \int_{\mathbb{R}^n}
		(-s)^l (-y)^\beta (\omega_{hj} u_h) (s,y)
	dyds.
\end{split}
\]
Moreover, from \eqref{int-I}, $r_{0,n},\ldots,r_{3,n}$ are split to
\[
\begin{split}
	r_{0,n} (t)
	=&
	- \sum_{|\alpha| = n+2} \sum_{k=1}^n \frac{\nabla^\alpha R_k (-\Delta)^{-1/2} G(t)}{\alpha!} \int_{\mathbb{R}^n} (-y)^\alpha \omega_{0kj} (y) dy
	+
	r_{0,n+1} (t),\\
%\end{split}
%\]
%\[
%\begin{split}
%%	%
	r_{1,n} (t)
	=&
	- \sum_{k,h=1}^n \int_0^t \int_{\mathbb{R}^n}
		\biggl( R_k R_j G(t-s,x-y) - \sum_{2l+|\beta|=0}^{n+1} \frac{\partial_t^l \nabla^\beta R_k R_j G(t,x)}{l!\beta!} (-s)^l (-y)^\beta \biggr)
		\mathcal{I}_{hk;n+3} (s,y)
	dyds\\
	&\hspace{-10mm}+
	U_{j;n+1}^S (t)
	-
	\sum_{2l+|\beta|=n+1} \sum_{k,h=1}^n \frac{\partial_t^l \nabla^\beta R_k R_j G(t,x)}{l!\beta!}
	\int_0^t s^l (1+s)^{-\frac{n}2 - \frac32 + \frac{|\beta|}2} ds
	\int_{\mathbb{R}^n} (-1)^l (-y)^\beta \mathcal{I}_{hk;n+3} (1,y) dy\\
	&+
	r_{1,n+1} (t),\\
%\end{split}
%\]
%\[
%\begin{split}
%	%
	r_{2,n} (t)
	=&
	- \sum_{h=1}^n \int_0^t \int_{\mathbb{R}^n} \biggl(
		G(t-s,x-y) - \sum_{2l + |\beta| = 0}^{n+1} \frac{\partial_t^l \nabla^\beta G(t,x)}{l!\beta!} (-s)^l (-y)^\beta
	\biggr)
	\mathcal{I}_{hj;n+3} (s,y)
	dyds\\
%\end{split}
%\]
%\[
%\begin{split}
	&-
	\sum_{2l+|\beta| = n+1} \frac{\partial_t^l \nabla^\beta G(t,x)}{l!\beta!} \sum_{h=1}^n \int_0^\infty \int_{\mathbb{R}^n}
		(-s)^l (-y)^\beta \left( \omega_{hj} u_h (s,y) - \mathcal{I}_{hj;n+3} (1+s,y) \right)
	dyds\\
%\end{split}
%\]
%\[
%\begin{split}
	&- 
	\sum_{2l+|\beta|=n+1} \frac{\partial_t^l \nabla^\beta G(t,x)}{l!\beta!} \int_0^t s^l (1+s)^{-\frac{n}2 - \frac32 + \frac{|\beta|}2} ds \sum_{h=1}^n \int_{\mathbb{R}^n}
		(-1)^l (-y)^\beta \mathcal{I}_{hj;n+3} (1,y)
	dy\\
	&+ r_{2,n+1} (t),\\
%\end{split}
%\]
%\[
%\begin{split}
%	%
	r_{3,n} (t)
	=&
	\sum_{2l+|\beta|=1}^n \frac{\partial_t^l \nabla^\beta G(t)}{l!\beta!} \sum_{h=1}^n \int_t^\infty \int_{\mathbb{R}^n}
		(-s)^l (-y)^\beta \mathcal{I}_{hj;n+3} (s,y)
	dyds
	+ r_{3,n+1} (t)\\
	=&
	V_{j;n+1} (t)
%	-\sum_{2l+|\beta|=1}^n
%	\frac{2 t^{-\frac{n}2-\frac12+l+\frac{|\beta|}2}}{n+1-2l-|\beta|}
%	\sum_{h=1}^n \frac{\partial_t^l \nabla^\beta G(t)}{l!\beta!}
%	 \int_{\mathbb{R}^n}
%		(-1)^l (-y)^\beta \left( \Omega_{hj;2}^\infty U_{h;1}^\infty \right) (1,y)
%	dy
	+ r_{3,n+1} (t),
\end{split}
\]
where
\[
\begin{split}
	r_{0,n+1} (t)
	=&
	- \sum_{k=1}^n R_k (-\Delta)^{-1/2} G(t) * \omega_{0kj}
	+ \sum_{|\alpha| = 0}^{n+2} \sum_{k=1}^n \frac{\nabla^\alpha R_k (-\Delta)^{-1/2} G(t)}{\alpha!} \int_{\mathbb{R}^n} (-y)^\alpha \omega_{0kj} (y) dy,
\end{split}
\]
\[
\begin{split}
	r_{1,n+1} (t)
	=&
	- \sum_{k,h=1}^n \int_0^t \int_{\mathbb{R}^n}
		\biggl( R_k R_j G (t-s,x-y) - \sum_{2l+|\beta| = 0}^{n+1} \frac{\partial_t^l \nabla^\beta R_k R_j G (t,x)}{l!\beta!} (-s)^l (-y)^\beta \biggr)\\
		&\hspace{15mm}\times
		\left( \omega_{hk} u_h - \mathcal{I}_{hk;n+3} \right) (s,y)
	dyds,\\
%\end{split}
%\]
%\[
%\begin{split}
%	%
	%
	r_{2,n+1} (t)
	=&
	- \sum_{h=1}^n \int_0^t \int_{\mathbb{R}^n} \biggl(
		G (t-s,x-y) - \sum_{2l+|\beta|=0}^{n+1} \frac{\partial_t^l \nabla^\beta G(t,x)}{l!\beta!} (-s)^l (-y)^\beta \biggr)
%		&\hspace{15mm}\times
		\left( \omega_{hj} u_h - \mathcal{I}_{hj;n+3} \right) (s,y)
	dyds\\
%	,\\
%	%
%	\rho_{2,2} (t)
%	=&
	&+ \sum_{2l+|\beta|=n+1} \frac{\partial_t^l \nabla^\beta G (t,x)}{l!\beta!} \sum_{h=1}^n \int_t^\infty \int_{\mathbb{R}^n}
		(-s)^l (-y)^\beta
%		&\hspace{15mm}\times
		\left( \omega_{hj} u_h (s,y) - \mathcal{I}_{hj;n+3} (1+s,y) \right)
	dyds,\\
%\end{split}
%\]
%\[
%\begin{split}
%	%
	r_{3,n+1} (t)
	=&
	\sum_{2l+|\beta|=1}^n \frac{\partial_t^l \nabla^\beta G(t)}{l!\beta!} \sum_{h=1}^n \int_t^\infty \int_{\mathbb{R}^n}
		(-s)^l (-y)^\beta \left( \omega_{hj} u_h - \mathcal{I}_{hj;n+3} \right) (s,y)
	dyds.
\end{split}
\]
Therefore
\[
	r_{0,n} + \cdots + r_{3,n}
	=
	K_{j;n+1} + U_{j;n+1} + U_{j;n+1}^S + V_{j;n+1} + J_{j;n+1}
	+
	r_{0,n+1} + \cdots + r_{3,n+1}.
\]
We repeat this procedure, then
\[
\begin{split}
	r_{1,n+1} (t)
	=&
	- \sum_{2l+|\beta|=n+2} \sum_{k,h=1}^n \frac{\partial_t^l \nabla^\beta R_k R_j G(t)}{l!\beta!} \int_0^t s^l (1+s)^{-\frac{n}2 - 2 + \frac{|\beta|}2} ds
%	&\hspace{40mm}\times
	\int_{\mathbb{R}^n} (-1)^l (-y)^\beta \mathcal{I}_{hk;n+4} (1,y) dy\\
%\end{split}
%\]
%\[
%\begin{split}
	&+
	U_{j;n+2}^S (t)\\
%\end{split}
%\]
%\[
%\begin{split}
%	\sum_{2l+|\beta|=n+2} \sum_{k,h=1}^n \frac{\partial_t^l \nabla^\beta R_k R_j G(t)}{l!\beta!}
%	\int_0^t \int_{\mathbb{R}^n}
%		(-s)^l (-y)^\beta \bigl( \omega_{hk} u_h (s,y) - \mathcal{I}_{hk;n+3}^\infty (s,y)\\
%		&\hspace{15mm} - \mathcal{I}_{hk;n+4}^\infty (1+s,y) \bigr)
%	dyds\\
%\end{split}
%\]
%\[
%\begin{split}
	&-
	\sum_{k,h=1}^n \int_0^t \int_{\mathbb{R}^n}
		\biggl( R_k R_j G(t-s,x-y) - \sum_{2l+|\beta|=0}^{n+2} \frac{\partial_t^l \nabla^\beta R_k R_j G(t,x)}{l!\beta!} (-s)^l (-y)^\beta \biggr)
%		&\hspace{15mm}\times
		\mathcal{I}_{hk;n+4} (s,y)
	dyds\\
	&+
	r_{1,n+2} (t),\\
%\end{split}
%\]
%\[
%\begin{split}
%%	%
	r_{2,n+1} (t)
	=&
	- \sum_{2l + |\beta| = n+2} \frac{\partial_t^l \nabla^\beta G (t)}{l!\beta!}
	\int_0^t s^l (1+s)^{-\frac{n}2 - 2 + \frac{|\beta|}2} ds
	 \sum_{h=1}^n 
	 \int_{\mathbb{R}^n}
		(-1)^l (-y)^\beta \mathcal{I}_{hj;n+4} (1,y)
	dy\\
%\end{split}
%\]
%\[
%\begin{split}
	&-
	\sum_{2l+|\beta| = n+2} \frac{\partial_t^l \nabla^\beta G(t)}{l!\beta!}
	\sum_{h=1}^n \int_0^\infty \int_{\mathbb{R}^n}
		(-s)^l (-y)^\beta \bigl( \omega_{hj} u_h (s,y) - \mathcal{I}_{hj;n+3} (s,y)\\
		&\hspace{25mm}
		- \mathcal{I}_{hj;n+4} (1+s,y) \bigr)
	dyds\\
%\end{split}
%\]
%\[
%\begin{split}
	&-
	\int_0^t \int_{\mathbb{R}^n}
		\biggl( G(t-s,x-y) - \sum_{2l+|\beta|=0}^{n+2} \frac{\partial_t^l \nabla^\beta G(t,x)}{l!\beta!} (-s)^l (-y)^\beta \biggr)
		\sum_{h=1}^n \mathcal{I}_{hj;n+4} (s,y)
	dyds\\
%\end{split}
%\]
%\[
%\begin{split}
	&+
	\sum_{2l+|\beta|=n+1} \frac{2t^{-\frac12} \partial_t^l \nabla^\beta G(t)}{l!\beta!}
	\sum_{h=1}^n \int_{\mathbb{R}^n}
		(-1)^l (-y)^\beta \mathcal{I}_{hj;n+4} (1,y)
	dy
	+r_{2,n+2} (t),
\end{split}
\]
where
\[
\begin{split}
	r_{1,n+2} (t)
	=&
	- \sum_{k,h=1}^n \int_0^t \int_{\mathbb{R}^n}
		\biggl( R_k R_j G(t-s,x-y) - \sum_{2l+|\beta|=0}^{n+2} \frac{\partial_t^l \nabla^\beta R_k R_j G(t,x)}{l!\beta!} (-s)^l (-y)^\beta \biggr)\\
		&\hspace{25mm}\times
		\left( \omega_{hk} u_h - \mathcal{I}_{hk;n+3} - \mathcal{I}_{hk;n+4} \right) (s,y)
	dyds,\\
%\end{split}
%\]
%\[
%\begin{split}
%%	%
	r_{2,n+2} (t)
	=&
	\sum_{2l+|\beta|=n+2} \frac{\partial_t^l \nabla^\beta G(t)}{l!\beta!} \sum_{h=1}^n \int_t^\infty \int_{\mathbb{R}^n}
		(-s)^l (-y)^\beta \bigl( \omega_{hj} u_h (s,y) - \mathcal{I}_{hj;n+3} (s,y)\\
		&\hspace{30mm}
		- \mathcal{I}_{hj;n+4} (1+s,y) \bigr)
	dyds
\end{split}
\]
\[
\begin{split}
	&-
	\int_0^t \int_{\mathbb{R}^n}
		\biggl( G(t-s,x-y) - \sum_{2l+|\beta|=0}^{n+2} \frac{\partial_t^l \nabla^\beta G(t,x)}{l!\beta!} (-s)^l (-y)^\beta \biggr)\\
		&\hspace{5mm}\times
		\sum_{h=1}^n \left( \omega_{hj} u_h - \mathcal{I}_{hj;n+3} - \mathcal{I}_{hj;n+4} \right) (s,y)
	dyds\\
%\end{split}
%\]
%\[
%\begin{split}
%	&{\color{red} +
%	\sum_{2l+|\beta|=n+2} \frac{\partial_t^l \nabla^\beta G(t)}{l!\beta!} \sum_{h=1}^n \int_t^\infty \int_{\mathbb{R}^n}
%		(-s)^l (-y)^\beta \bigl( \omega_{hj} u_h(s,y) - \mathcal{I}_{hj;n+3} (s,y)}\\
%		&\hspace{15mm}{\color{red} - \mathcal{I}_{hj;n+4} (1+s,y) \bigr)
%	dyds}\\
	&+
	\sum_{2l+|\beta|=n+1} \frac{\partial_t^l \nabla^\beta G(t)}{l!\beta!} \sum_{h=1}^n \int_t^\infty \int_{\mathbb{R}^n}
		(-s)^l (-y)^\beta \bigl( \omega_{hj} u_h - \mathcal{I}_{hj;n+3} - \mathcal{I}_{hj;n+4} \bigr) (s,y)
	dyds\\
%\end{split}
%\]
%\[
%\begin{split}
	&-
	\sum_{2l+|\beta|=n+1} \frac{\partial_t^l \nabla^\beta G(t)}{l!\beta!}
	\int_t^\infty s^l \bigl( (1+s)^{-\frac{n}2-\frac32+\frac{|\beta|}2} - s^{-\frac{n}2-\frac32+\frac{|\beta|}2} \bigr) ds\\
	&\hspace{25mm}\times
	\sum_{h=1}^n \int_{\mathbb{R}^n}
		(-1)^l (-y)^\beta \mathcal{I}_{hj;n+3} (1,y)
	dy.
\end{split}
\]
For the last term, \eqref{int-I} leads that
\[
\begin{split}
	r_{3,n+1} (t)
	=&
	\sum_{2l+|\beta| = 1}^n
	\frac{2t^{-\frac{n}2-1+l+\frac{|\beta|}2}}{n+2-2l-|\beta|}
	\frac{\partial_t^l \nabla^\beta G(t)}{l!\beta!}
%	&\hspace{5mm}\times
	\sum_{h=1}^n \int_{\mathbb{R}^n}
		(-1)^l (-y)^\beta \mathcal{I}_{hj;n+4} (1,y)
	dy
	+
	r_{3,n+2} (t),
\end{split}
\]
where
\[
\begin{split}
	r_{3,n+2} (t)
	=&
	\sum_{2l+|\beta|=1}^n \frac{\partial_t^l \nabla^\beta G (t)}{l!\beta!} \sum_{h=1}^n \int_t^\infty \int_{\mathbb{R}^n}
		(-s)^l (-y)^\beta \bigl( \omega_{hj} u_h - \mathcal{I}_{hj;n+3} - \mathcal{I}_{hj;n+4} \bigr) (s,y)
	dyds.
\end{split}
\]
Therefore
\[
\begin{split}
	r_{0,n+1} + \cdots + r_{3,n+1}
	=&
	K_{j;n+2} + U_{j;n+2} + U_{j;n+2}^S + V_{j;n+2} + J_{j;n+2}
	+ r_{0,n+2} + \cdots + r_{3,n+2}.
\end{split}
\]
We expand the first term of $r_{2,n+2}$, then, from \eqref{int-I},
\[
\begin{split}
	&\sum_{2l+|\beta|=n+2} \frac{\partial_t^l \nabla^\beta G (t)}{l!\beta!} \sum_{h=1}^n \int_t^\infty \int_{\mathbb{R}^n}
		(-s)^l (-y)^\beta \bigl( \omega_{hj} u_h (s,y) - \mathcal{I}_{hj;n+3} (s,y) - \mathcal{I}_{hj;n+4} (1+s,y) \bigr)
	dyds\\
	=&
	\sum_{2l+|\beta|=n+2} \frac{\partial_t^l \nabla^\beta G (t)}{l!\beta!} \sum_{h=1}^n \int_t^\infty \int_{\mathbb{R}^n}
		(-s)^l (-y)^\beta \mathcal{I}_{hj;n+5} (s,y)
	dyds\\
	&-
	\sum_{2l+|\beta|=n+2} \frac{\partial_t^l \nabla^\beta G (t)}{l!\beta!} \sum_{h=1}^n \int_t^\infty \int_{\mathbb{R}^n}
		(-s)^l (-y)^\beta \bigl( \mathcal{I}_{hj;n+4} (1+s,y) - \mathcal{I}_{hj;n+4} (s,y) \bigr)
	dyds\\
%\end{split}
%\]
%\[
%\begin{split}
%%	%
	&+
	\sum_{2l+|\beta|=n+2} \frac{\partial_t^l \nabla^\beta G (t)}{l!\beta!} \sum_{h=1}^n \int_t^\infty \int_{\mathbb{R}^n}
		(-s)^l (-y)^\beta \biggl( \omega_{hj} u_h - \sum_{p=3}^5 \mathcal{I}_{hj;n+p} \biggr) (s,y)
	dyds\\
%\end{split}
%\]
%\[
%\begin{split}
%%	%
	=&
	\sum_{2l+|\beta|=n+2} \frac{2t^{-\frac12} \partial_t^l \nabla^\beta G (t)}{l!\beta!} \sum_{h=1}^n \int_{\mathbb{R}^n}
		(-1)^l (-y)^\beta \mathcal{I}_{hj;n+5} (1,y)
	dy\\
	&-
	\sum_{2l+|\beta|=n+2} \frac{\partial_t^l \nabla^\beta G(t)}{l!\beta!} \int_t^\infty s^l \left( (1+s)^{-\frac{n}2-2+\frac{|\beta|}2} - s^{-\frac{n}2-2+\frac{|\beta|}2} \right) ds
	\sum_{h=1}^n \int_{\mathbb{R}^n}
		(-1)^l (-y)^\beta \mathcal{I}_{hj;n+4} (1,y)
	dy\\
	&+
	\sum_{2l+|\beta|=n+2} \frac{\partial_t^l \nabla^\beta G (t)}{l!\beta!} \sum_{h=1}^n \int_t^\infty \int_{\mathbb{R}^n}
		(-s)^l (-y)^\beta \biggl( \omega_{hj} u_h - \sum_{p=3}^5 \mathcal{I}_{hj;n+p} \biggr) (s,y)
	dyds.
\end{split}
\]
For the second term of $r_{2,n+2}$, we see
\[
\begin{split}
	&- \int_0^t \int_{\mathbb{R}^n}
		\biggl( G(t-s,x-y) - \sum_{2l+|\beta|=0}^{n+2} \frac{\partial_t^l \nabla^\beta G(t,x)}{l!\beta!} \biggr)
		\sum_{h=1}^n \biggl( \omega_{hj} u_h - \sum_{p=3}^4 \mathcal{I}_{hj;n+p} \biggr) (s,y)
	dyds\\
	=&
	- \sum_{2l+|\beta|=n+3} \frac{\partial_t^l \nabla^\beta G(t)}{l!\beta!} \sum_{h=1}^n \int_0^t \int_{\mathbb{R}^n}
		(-s)^l (-y)^\beta \biggl( \omega_{hj} u_h - \sum_{p=3}^4 \mathcal{I}_{hj;n+p} \biggr) (s,y)
	dyds\\
	&-
	\int_0^t \int_{\mathbb{R}^n}
		\biggl( G(t-s,x-y) - \sum_{2l+|\beta|=0}^{n+3} \frac{\partial_t^l \nabla^\beta G(t,x)}{l!\beta!} (-s)^l (-y)^\beta \biggr)
		\sum_{h=1}^n \biggl( \omega_{hj} u_h - \sum_{p=3}^4 \mathcal{I}_{hj;n+p} \biggr) (s,y)
	dyds\\
%\end{split}
%\]
%\[
%\begin{split}
	=&
	- \sum_{2l+|\beta| = n+3} \frac{\partial_t^l \nabla^\beta G(t)}{l!\beta!}
	\int_0^t
		s^l (1+s)^{-\frac{n}2 - \frac52 + \frac{|\beta|}2}
	ds
	\sum_{h=1}^n \int_{\mathbb{R}^n}
		(-1)^l (-y)^\beta \mathcal{I}_{hj;n+5} (1,y)
	dy\\
%\end{split}
%\]
%\[
%\begin{split}
	&-
	\sum_{2l+|\beta|=n+3} \frac{\partial_t^l \nabla^\beta G(t)}{l!\beta!} \sum_{h=1}^n \int_0^\infty \int_{\mathbb{R}^n}
		(-s)^l (-y)^\beta \biggl( \omega_{hj} u_h (s,y) - \sum_{p=3}^4 \mathcal{I}_{hj;n+p} (s,y) - \mathcal{I}_{hj;n+5} (1+s,y) \biggr)
	dyds\\
%\end{split}
%\]
%\[
%\begin{split}
	&-
	\int_0^t \int_{\mathbb{R}^n}
		\biggl( G(t-s,x-y) - \sum_{2l+|\beta|=0}^{n+3} \frac{\partial_t^l \nabla^\beta G(t,x)}{l!\beta!} (-s)^l (-y)^\beta \biggr)
		\sum_{h=1}^n \mathcal{I}_{hj;n+5} (s,y)
	dyds\\
%\end{split}
%\]
%\[
%\begin{split}
	&-
	\int_0^t \int_{\mathbb{R}^n}
		\biggl( G(t-s,x-y) - \sum_{2l+|\beta|=0}^{n+3} \frac{\partial_t^l \nabla^\beta G(t,x)}{l!\beta!} (-s)^l (-y)^\beta \biggr)
		\sum_{h=1}^n \biggl( \omega_{hj} u_h - \sum_{p=3}^5 \mathcal{I}_{hj;n+p} \biggr) (s,y)
	dyds\\
	&+
	\sum_{2l+|\beta|=n+3} \frac{\partial_t^l \nabla^\beta G(t)}{l!\beta!} \sum_{h=1}^n \int_t^\infty \int_{\mathbb{R}^n}
		(-s)^l (-y)^\beta \biggl( \omega_{hj} u_h (s,y) - \sum_{p=3}^4 \mathcal{I}_{hj;n+p} (s,y) - \mathcal{I}_{hj;n+5} (1+s,y) \biggr)
	dyds.
\end{split}
\]
For the third term of $r_{2,n+2}$,
\[
\begin{split}
	&\sum_{2l+|\beta|=n+1} \frac{\partial_t^l \nabla^\beta G(t)}{l!\beta!} \sum_{h=1}^n \int_t^\infty \int_{\mathbb{R}^n}
		(-s)^l (-y)^\beta \biggl( \omega_{hj} u_h - \sum_{p=3}^4 \mathcal{I}_{hj;n+p} \biggr) (s,y)
	dyds\\
	=&
	\sum_{2l+|\beta|=n+1} \frac{t^{-1} \partial_t^l \nabla^\beta G(t)}{l!\beta!} \sum_{h=1}^n \int_{\mathbb{R}^n}
		(-1)^l (-y)^\beta \mathcal{I}_{hj;n+5} (1,y)
	dy\\
	&+ 
	\sum_{2l+|\beta|=n+1} \frac{\partial_t^l \nabla^\beta G(t)}{l!\beta!} \sum_{h=1}^n \int_t^\infty \int_{\mathbb{R}^n}
		(-s)^l (-y)^\beta \biggl( \omega_{hj} u_h - \sum_{p=3}^5 \mathcal{I}_{hj;n+p} \biggr) (s,y)
	dyds.
\end{split}
\]
The last term of $r_{2,n+2}$ is $\tilde{V}_{j;n+3}$.
Hence
\[
\begin{split}
	&r_{2,n+2} (t)
	=
	\sum_{2l+|\beta|=n+2} \frac{2t^{-\frac12} \partial_t^l \nabla^\beta G (t)}{l!\beta!} \sum_{h=1}^n \int_{\mathbb{R}^n}
		(-1)^l (-y)^\beta \mathcal{I}_{hj;n+5} (1,y)
	dy\\
%\end{split}
%\]
%\[
%\begin{split}
	&-
	\sum_{2l+|\beta| = n+3} \frac{\partial_t^l \nabla^\beta G(t)}{l!\beta!}
	\int_0^t
		s^l (1+s)^{-\frac{n}2 - \frac52 + \frac{|\beta|}2}
	ds
	\sum_{h=1}^n \int_{\mathbb{R}^n}
		(-1)^l (-y)^\beta \mathcal{I}_{hj;n+5} (1,y)
	dy\\
%\end{split}
%\]
%\[
%\begin{split}
	&-
	\sum_{2l+|\beta|=n+3} \frac{\partial_t^l \nabla^\beta G(t)}{l!\beta!} \sum_{h=1}^n \int_0^\infty \int_{\mathbb{R}^n}
		(-s)^l (-y)^\beta \biggl( \omega_{hj} u_h (s,y) - \sum_{p=3}^4 \mathcal{I}_{hj;n+p} (s,y) - \mathcal{I}_{hj;n+5} (1+s,y) \biggr)
	dyds\\
%\end{split}
%\]
%\[
%\begin{split}
	&-
	\int_0^t \int_{\mathbb{R}^n}
		\biggl( G(t-s,x-y) - \sum_{2l+|\beta|=0}^{n+3} \frac{\partial_t^l \nabla^\beta G(t,x)}{l!\beta!} (-s)^l (-y)^\beta \biggr)
		\sum_{h=1}^n \mathcal{I}_{hj;n+5} (s,y)
	dyds\\
%\end{split}
%\]
%\[
%\begin{split}
	&+
	\sum_{2l+|\beta|=n+1} \frac{t^{-1}\partial_t^l \nabla^\beta G(t)}{l!\beta!} \sum_{h=1}^n \int_{\mathbb{R}^n}
		(-1)^l (-y)^\beta \mathcal{I}_{hj;n+5} (1,y)
	dy
%\end{split}
%\]
%\[
%\begin{split}
	+
	\tilde{V}_{j;n+3} (t)
	+
	r_{2,n+3} (t),
\end{split}
\]
where
\[
\begin{split}
	&r_{2,n+3} (t)
	=
%\end{split}
%\]
%\[
%\begin{split}
	\sum_{2l+|\beta|=n+3} \frac{\partial_t^l \nabla^\beta G(t)}{l!\beta!} \sum_{h=1}^n \int_t^\infty \int_{\mathbb{R}^n}
		(-s)^l (-y)^\beta \bigl( \omega_{hj} u_h (s,y) - \sum_{p=3}^4 \mathcal{I}_{hj;n+p} (s,y)\\
		&\hspace{35mm} - \mathcal{I}_{hj;n+5} (1+s,y) \bigr) (s,y)
	dyds\\
%\end{split}
%\]
%\[
%\begin{split}
	&-
	\int_0^t \int_{\mathbb{R}^n}
		\biggl( G(t-s,x-y) - \sum_{2l+|\beta|=0}^{n+3} \frac{\partial_t^l \nabla^\beta G(t,x)}{l!\beta!} (-s)^l (-y)^\beta \biggr)
		\sum_{h=1}^n \biggl( \omega_{hj} u_h - \sum_{p=3}^5 \mathcal{I}_{hj;n+p} \biggr) (s,y)
	dyds\\
%	&-
%	\sum_{2l+|\beta|=n+2} \frac{\partial_t^l \nabla^\beta G (t)}{l!\beta!} \sum_{h=1}^n \int_t^\infty \int_{\mathbb{R}^n}
%		(-s)^l (-y)^\beta \biggl( \omega_{hj} u_h - \sum_{p=3}^5 \mathcal{I}_{hj;n+p}^\infty \biggr) (s,y)
%	dyds\\
%\end{split}
%\]
%\[
%\begin{split}
	&+ 
	\sum_{2l+|\beta|=n+1}^{n+2} \frac{\partial_t^l \nabla^\beta G(t)}{l!\beta!} \sum_{h=1}^n \int_t^\infty \int_{\mathbb{R}^n}
		(-s)^l (-y)^\beta \biggl( \omega_{hj} u_h - \sum_{p=3}^5 \mathcal{I}_{hj;n+p} \biggr) (s,y)
	dyds\\
	&-
	\sum_{2l+|\beta|=n+2} \frac{\partial_t^l \nabla^\beta G(t)}{l!\beta!} \int_t^\infty s^l \left( (1+s)^{-\frac{n}2-2+\frac{|\beta|}2} - s^{-\frac{n}2-2+\frac{|\beta|}2} \right) ds
	\sum_{h=1}^n \int_{\mathbb{R}^n}
		(-1)^l (-y)^\beta \mathcal{I}_{hj;n+4} (1,y)
	dy.
\end{split}
\]
Similarly,
\[
\begin{split}
	&r_{1,n+2} (t)\\
	=&
	- \sum_{2l+|\beta| = n+3} \sum_{k,h=1}^n \frac{\partial_t^l \nabla^\beta R_k R_j G(t)}{l!\beta!}
	\int_0^t
		s^l (1+s)^{-\frac{n}2 - \frac52 + \frac{|\beta|}2}
	ds
	\int_{\mathbb{R}^n}
		(-1)^l (-y)^\beta \mathcal{I}_{hk;n+5} (1,y)
	dy
	+
	U_{j;n+3}^S (t)\\
%	\sum_{2l+|\beta|=n+3} \sum_{k,h=1}^n \frac{\partial_t^l \nabla^\beta R_k R_j G(t)}{l!\beta!} \int_0^\infty \int_{\mathbb{R}^n}
%		(-s)^l (-y)^\beta \bigl( \omega_{hk} u_h (s,y) - \sum_{p=3}^4 \mathcal{I}_{hk;n+p}^\infty (s,y)\\
%		&\hspace{25mm}- \mathcal{I}_{hk;n+5}^\infty (1+s,y) \bigr)
%	dyds\\
%\end{split}
%\]
%\[
%\begin{split}
%	%
	&-
	\sum_{k,h=1}^n \int_0^t \int_{\mathbb{R}^n}
		\biggl( R_k R_j G(t-s,x-y) - \sum_{2l+|\beta|=0}^{n+3} \frac{\partial_t^l \nabla^\beta R_k R_j G(t,x)}{l!\beta!} (-s)^l (-y)^\beta \biggr)
		\mathcal{I}_{hk;n+5} (s,y)
	dyds\\
%\end{split}
%\]
%\[
%\begin{split}
%	%
%\end{split}
%\]
%\[
%\begin{split}
%	%
	&+ r_{1,n+3} (t),
\end{split}
\]
where
\[
\begin{split}
	&r_{1,n+3} (t)
	=
	- \sum_{k,h=1}^n \int_0^t \int_{\mathbb{R}^n}
		\Bigl( R_k R_j G(t-s,x-y) - \sum_{2l+|\beta|=0}^{n+3} \frac{\partial_t^l \nabla^\beta R_k R_j G(t,x)}{l!\beta!} (-s)^l (-y)^\beta \Bigr)\\
		&\hspace{15mm}\times
		\Bigl( \omega_{hk} u_h - \sum_{p=3}^5 \mathcal{I}_{hk;n+p} \Bigr) (s,y)
	dyds.
\end{split}
\]
At last, from \eqref{int-I},
\[
\begin{split}
	r_{3, n+2} (t)
	=&
	\sum_{2l+|\beta|=1}^n \frac{2 t^{-\frac{n}2 - \frac32 + l + \frac{|\beta|}2}}{n+3 - 2l - |\beta|} \frac{\partial_t^l \nabla^\beta G(t)}{l!\beta!} \sum_{h=1}^n \int_{\mathbb{R}^n}
		(-1)^l (-y)^\beta \mathcal{I}_{hj,n+5} (1,y)
	dy
	+ r_{3,n+3} (t),
\end{split}
\]
where
\[
\begin{split}
	r_{3,n+3} (t)
	=&
	\sum_{2l + |\beta| = 1}^n \frac{\partial_t^l \nabla^\beta G(t)}{l!\beta!} \int_t^\infty \int_{\mathbb{R}^n}
		(-s)^l (-y)^\beta \biggl( \omega_{hj} u_h - \sum_{p=3}^5 \mathcal{I}_{hj,n+p} \biggr) (s,y)
	dyds.
\end{split}
\]
Thus
\[
\begin{split}
	r_{0,n+2} + \cdots + r_{3,n+2}
	=&
	K_{j;n+3} + U_{j;n+3} + U_{j;n+3}^S + V_{j;n+3} + J_{j;n+3} + \tilde{V}_{j;n+3}
	+ r_{0,n+3} + \cdots + r_{3,n+3}.
\end{split}
\]
Generally, for $1 \le m \le n$, let
\[
\begin{split}
	r_{0,n+m} (t)
	=&
	- \sum_{k=1}^n R_k (-\Delta)^{-1/2} G(t) * \omega_{0kj}
	+
	\sum_{|\alpha|=0}^{n+m+1} \sum_{k=1}^n \frac{\nabla^\alpha R_k (-\Delta)^{-1/2} G(t)}{\alpha!} \int_{\mathbb{R}^n} (-y)^\alpha \omega_{0kj} (y) dy,
\end{split}
\]
\[
\begin{split}
	r_{1,n+m} (t)
	=&
	- \sum_{k,h=1}^n \int_0^t \int_{\mathbb{R}^n}
		\biggl( R_k R_j G(t-s,x-y) - \sum_{2l+|\beta|=0}^{n+m} \frac{\partial_t^l \nabla^\beta R_k R_j G(t,x)}{l!\beta!} (-s)^l (-y)^\beta \biggr)\\
		&\hspace{15mm}\times
		\biggl( \omega_{hk} u_h - \sum_{p=3}^{m+2} \mathcal{I}_{hk;n+p} \biggr) (s,y)
	dyds,\\
%\end{split}
%\]
%\[
%\begin{split}
%	%
	r_{2,n+m} (t)
	=&
	\sum_{2l+|\beta| = n+m} \frac{\partial_t^l \nabla^\beta G(t)}{l!\beta!} \sum_{h=1}^n \int_t^\infty \int_{\mathbb{R}^n}
		(-s)^l (-y)^\beta \bigl( \omega_{hj} u_h (s,y) - \sum_{p=3}^{m+1} \mathcal{I}_{hj;n+p} (s,y)\\
		&\hspace{15mm}-\mathcal{I}_{hj;n+m+2} (1+s,y) \bigr)
	dyds\\
%\end{split}
%\]
%\[
%\begin{split}
	&-
	\int_0^t \int_{\mathbb{R}^n}
		\biggl( G(t-s,x-y) - \sum_{2l+|\beta|=0}^{n+m} \frac{\partial_t^l \nabla^\beta G(t,x)}{l!\beta!} (-s)^l (-y)^\beta \biggr)\\
		&\hspace{5mm}\times
		\sum_{h=1}^n \biggl( \omega_{hj} u_h - \sum_{p=3}^{m+2} \mathcal{I}_{hj;n+p} \biggr) (s,y)
	dyds\\
%\end{split}
%\]
%\[
%\begin{split}
	&+
	\sum_{2l+|\beta|=n+1}^{n+m-1} \frac{\partial_t^l \nabla^\beta G(t)}{l!\beta!} \sum_{h=1}^n \int_t^\infty \int_{\mathbb{R}^n}
		(-s)^l (-y)^\beta \biggl( \omega_{hj} u_h - \sum_{p=3}^{m+2} \mathcal{I}_{hj;n+p} \biggr) (s,y)
	dyds\\
%\end{split}
%\]
%\[
%\begin{split}
	&-
	\sum_{2l+|\beta|=n+m-1} \frac{\partial_t^l \nabla^\beta G(t)}{l!\beta!}
	\int_t^\infty s^l \bigl( (1+s)^{-\frac{n}2 - \frac{m}2 - \frac12 + \frac{|\beta|}2} - s^{-\frac{n}2 - \frac{m}2 - \frac12 + \frac{|\beta|}2} \bigr) ds\\
	&\hspace{5mm}\times
	\sum_{h=1}^n \int_{\mathbb{R}^n} (-1)^l (-y)^\beta \mathcal{I}_{hj;n+m+1} (1,y) dy,\\
	r_{3,n+m} (t)
	=&
	\sum_{2l+|\beta|=1}^n \frac{\partial_t^l \nabla^\beta G(t)}{l!\beta!} \sum_{h=1}^n \int_t^\infty \int_{\mathbb{R}^n}
		(-s)^l (-y)^\beta \biggl( \omega_{hj} u_h - \sum_{p=3}^{m+2} \mathcal{I}_{hj;n+p} \biggr) (s,y)
	dyds,
\end{split}
\]
then, for $1 \le m \le n-1$,
\begin{equation}\label{kinou}
\begin{split}
	r_{0,n+m} + \cdots + r_{3,n+m}
	=&
	K_{j;n+m+1} + U_{j;n+m+1} + U_{j;n+m+1}^S + V_{j;n+m+1} + J_{j;n+3} + \tilde{V}_{j;n+m+1}\\
	&+ r_{0,n+m+1} + \cdots + r_{3,n+m+1}.
\end{split}
\end{equation}
We already confirmed it for $m = 1$ and $2$.
Inductively, for $3 \le m \le n-1$, we expand $r_{2,n+m}$, then, for the first and the second terms, we see from \eqref{int-I} that
\[
\begin{split}
	&\sum_{2l+|\beta| = n+m} \frac{\partial_t^l \nabla^\beta G(t)}{l!\beta!} \sum_{h=1}^n \int_t^\infty \int_{\mathbb{R}^n}
		(-s)^l (-y)^\beta \bigl( \omega_{hj} u_h (s,y) - \sum_{p=3}^{m+1} \mathcal{I}_{hj;n+p} (s,y)\\
		&\hspace{15mm}
		-\mathcal{I}_{hj;n+m+2} (1+s,y) \bigr)
	dyds\\
%\end{split}
%\]
%\[
%\begin{split}
	=
	&\sum_{2l+|\beta|=n+m} \frac{\partial_t^l \nabla^\beta G(t)}{l!\beta!} \sum_{h=1}^n \int_t^\infty \int_{\mathbb{R}^n}
		(-s)^l (-y)^\beta \mathcal{I}_{hj;n+m+3} (s,y)
	dyds\\
%\end{split}
%\]
%\[
%\begin{split}
	&-  \sum_{2l+|\beta|=n+m} \frac{\partial_t^l \nabla^\beta G(t)}{l!\beta!} \sum_{h=1}^n \int_t^\infty \int_{\mathbb{R}^n}
		(-s)^l (-y)^\beta \left( \mathcal{I}_{hj;n+m+2} (1+s,y) - \mathcal{I}_{hj;n+m+2} (s,y) \right)
	dyds\\
%\end{split}
%\]
%\[
%\begin{split}
	&+
	\sum_{2l+|\beta|=n+m} \frac{\partial_t^l \nabla^\beta G(t)}{l!\beta!} \sum_{h=1}^n \int_t^\infty \int_{\mathbb{R}^n}
		(-s)^l (-y)^\beta \biggl( \omega_{hj} u_h - \sum_{p=3}^{m+3} \mathcal{I}_{hj;n+p} \biggr) (s,y)
	dyds\\
%\end{split}
%\]
%\[
%\begin{split}
	=&
	\sum_{2l+|\beta|=n+m}
	\frac{2t^{-\frac12}\partial_t^l \nabla^\beta G (t)}{l!\beta!} \sum_{h=1}^n \int_{\mathbb{R}^n} (-1)^l (-y)^\beta \mathcal{I}_{hj;n+m+3} (1,y) dy
\end{split}
\]
\[
\begin{split}
	&-
	\sum_{2l+|\beta|=n+m} \frac{\partial_t^l \nabla^\beta G(t)}{l!\beta!}
	\int_t^\infty
		s^l \left( (1+s)^{-\frac{n}2 - \frac{m}2 -1 + \frac{|\beta|}2} - s^{-\frac{n}2 - \frac{m}2 -1 + \frac{|\beta|}2} \right)
	ds
	\sum_{h=1}^n \int_{\mathbb{R}^n}
		\mathcal{I}_{hj;n+m+2} (1,y)
	dy\\
%\end{split}
%\]
%\[
%\begin{split}
	&+
	\sum_{2l+|\beta|=n+m} \frac{\partial_t^l \nabla^\beta G(t)}{l!\beta!} \sum_{h=1}^n \int_t^\infty \int_{\mathbb{R}^n}
		(-s)^l (-y)^\beta \biggl( \omega_{hj} u_h - \sum_{p=3}^{m+3} \mathcal{I}_{hj;n+p} \biggr) (s,y)
	dyds
\end{split}
\]
and
\[
\begin{split}
	&- \int_0^t \int_{\mathbb{R}^n}
		\biggl( G(t-s,x-y) - \sum_{2l+|\beta|=0}^{n+m} \frac{\partial_t^l \nabla^\beta G(t,x)}{l!\beta!} (-s)^l (-y)^\beta \biggr)
%		&\hspace{5mm}\times
		\sum_{h=1}^n \biggl( \omega_{hj} u_h - \sum_{p=3}^{m+2} \mathcal{I}_{hj;n+p} \biggr) (s,y)
	dyds\\
%	=&
%	\sum_{2l+|\beta|=n+m+1} \frac{\partial_t^l \nabla^\beta G(t)}{l!\beta!} \int_0^t \int_{\mathbb{R}^n}
%		(-s)^l (-y)^\beta \biggl( \omega u - \sum_{p=3}^{m+2} \mathcal{I}_{n+p} \biggr) (s,y)
%	dyds\\
%	&+
%	\int_0^t \int_{\mathbb{R}^n}
%		 \biggl( G(t-s,x-y) - \sum_{2l+|\beta|=0}^{n+m+1} \frac{\partial_t^l \nabla^\beta G(t,x)}{l!\beta!} (-s)^l (-y)^\beta \biggr)
%		 \biggl( \omega u - \sum_{p=3}^{m+2} \mathcal{I}_{n+p} \biggr) (s,y)
%	dyds\\
	=&
	- \sum_{2l+|\beta|=n+m+1} \frac{\partial_t^l \nabla^\beta G(t)}{l!\beta!}
	\int_0^t
		s^l (1+s)^{-\frac{n}2 - \frac{m}2 - \frac32 + \frac{|\beta|}2}
	ds
	\sum_{h=1}^n \int_{\mathbb{R}^n}
		(-1)^l (-y)^\beta \mathcal{I}_{hj;n+m+3} (1,y)
	dy\\
%\end{split}
%\]
%\[
%\begin{split}
	&-
	\sum_{2l+|\beta|=n+m+1} \frac{\partial_t^l \nabla^\beta G(t)}{l!\beta!} \sum_{h=1}^n \int_0^\infty \int_{\mathbb{R}^n}
		(-s)^l (-y)^\beta \bigl( \omega_{hj} u_h (s,y) - \sum_{p=3}^{m+2} \mathcal{I}_{hj;n+p} (s,y)\\
		&\hspace{35mm} - \mathcal{I}_{hj;n+m+3} (1+s,y) \bigr)
	dyds\\
%\end{split}
%\]
%\[
%\begin{split}
	&-
	\int_0^t \int_{\mathbb{R}^n}
		\biggl( G(t-s,x-y) - \sum_{2l+|\beta|=0}^{n+m+1} \frac{\partial_t^l \nabla^\beta G(t,x)}{l!\beta!} (-s)^l (-y)^\beta \biggr)
		\sum_{h=1}^n \mathcal{I}_{hj;n+m+3} (s,y)
	dyds\\
%\end{split}
%\]
%\[
%\begin{split}
	&-
	\int_0^t \int_{\mathbb{R}^n}
		\biggl( G(t-s,x-y) - \sum_{2l+|\beta|=0}^{n+m+1} \frac{\partial_t^l \nabla^\beta G(t,x)}{l!\beta!} (-s)^l (-y)^\beta \biggr)
		\sum_{h=1}^n \biggl( \omega_{hj} u_h - \sum_{p=3}^{m+3} \mathcal{I}_{hj;n+p} \biggr) (s,y)
	dyds\\
%\end{split}
%\]
%\[
%\begin{split}
	&+
	\sum_{2l+|\beta| = n+m+1} \frac{\partial_t^l \nabla^\beta G(t)}{l!\beta!} \sum_{h=1}^n \int_t^\infty \int_{\mathbb{R}^n}
		(-s)^l (-y)^\beta \bigl( \omega_{hj} u_h (s,y) - \sum_{p=3}^{m+2} \mathcal{I}_{hj;n+p} (s,y)\\
		&\hspace{35mm} - \mathcal{I}_{hj;n+m+3} (1+s,y) \bigr)
	dyds.
\end{split}
\]
For the third term of $r_{2,n+m}$, we have that
\[
\begin{split}
	&\sum_{2l+|\beta|=n+1}^{n+m-1} \frac{\partial_t^l \nabla^\beta G(t)}{l!\beta!} \sum_{h=1}^n \int_t^\infty \int_{\mathbb{R}^n}
		(-s)^l (-y)^\beta \biggl( \omega_{hj} u_h - \sum_{p=3}^{m+2} \mathcal{I}_{hj;n+p} \biggr) (s,y)
	dyds\\
%\end{split}
%\]
%\[
%\begin{split}
	=&
	\sum_{2l+|\beta|=n+1}^{n+m-1} \frac{\partial_t^l \nabla^\beta G(t)}{l!\beta!} \sum_{h=1}^n \int_t^\infty \int_{\mathbb{R}^n}
		(-s)^l (-y)^\beta \mathcal{I}_{hj;n+m+3} (s,y)
	dyds\\
%\end{split}
%\]
%\[
%\begin{split}
	&+
	\sum_{2l+|\beta|=n+1}^{n+m-1} \frac{\partial_t^l \nabla^\beta G(t)}{l!\beta!} \sum_{h=1}^n \int_t^\infty \int_{\mathbb{R}^n}
		(-s)^l (-y)^\beta \biggl( \omega_{hj} u_h - \sum_{jp=3}^{m+3} \mathcal{I}_{hj;n+p} \biggr) (s,y)
	dyds\\
%\end{split}
%\]
%\[
%\begin{split}
	=&
	\sum_{2l+|\beta|=n+1}^{n+m-1} \frac{2t^{-\frac{n}2 - \frac{m}2 - \frac12 + l + \frac{|\beta|}2}}{n+m+1-2l-|\beta|} \frac{\partial_t^l \nabla^\beta G(t)}{l!\beta!} \sum_{h=1}^n \int_{\mathbb{R}^n}
		(-1)^l (-y)^\beta \mathcal{I}_{hj;n+m+3} (1,y)
	dy\\
%\end{split}
%\]
%\[
%\begin{split}
	&+
	\sum_{2l+|\beta|=n+1}^{n+m-1} \frac{\partial_t^l \nabla^\beta G(t)}{l!\beta!} \sum_{h=1}^n \int_t^\infty \int_{\mathbb{R}^n}
		(-s)^l (-y)^\beta \biggl( \omega_{hj} u_h - \sum_{jp=3}^{m+3} \mathcal{I}_{hj;n+p} \biggr) (s,y)
	dyds.
\end{split}
\]
The last term of $r_{2,n+m}$ is $\tilde{V}_{j;n+m+1}$.
The other terms $r_{1,n+m}$ and $r_{3,n+m}$ are expanded as
\[
\begin{split}
	&r_{1,n+m} (t)\\
%	&\sum_{k,h=1}^n \int_0^t \int_{\mathbb{R}^n}
%		\biggl( R_k R_j G(t-s,x-y) - \sum_{2l+|\beta|=0}^{n+m} \frac{\partial_t^l \nabla^\beta R_k R_j G(t,x)}{l!\beta!} (-s)^l (-y)^\beta \biggr)\\
%		&\hspace{5mm}\times
%		\biggl( \omega_{hk} u_h - \sum_{p=3}^{m+2} \mathcal{I}_{hk;n+p}^\infty \biggr) (s,y)
%	dyds\\
%\end{split}
%\]
%\[
%\begin{split}
%	=&
%	\sum_{2l+|\beta|=n+m+1} \frac{\partial_t^l \nabla^\beta G(t)}{l!\beta!} \int_0^t \int_{\mathbb{R}^n}
%		(-s)^l (-y)^\beta \biggl( \omega u - \sum_{p=3}^{m+2} \mathcal{I}_{n+p} \biggr) (s,y)
%	dyds\\
%	&+
%	\int_0^t \int_{\mathbb{R}^n}
%		 \biggl( G(t-s,x-y) - \sum_{2l+|\beta|=0}^{n+m+1} \frac{\partial_t^l \nabla^\beta G(t,x)}{l!\beta!} (-s)^l (-y)^\beta \biggr)
%		 \biggl( \omega u - \sum_{p=3}^{m+2} \mathcal{I}_{n+p} \biggr) (s,y)
%	dyds\\
	=&
	- \sum_{2l+|\beta|=n+m+1} \sum_{k,h=1}^n \frac{\partial_t^l \nabla^\beta R_k R_j G(t)}{l!\beta!}
	\int_0^t
		s^l (1+s)^{-\frac{n}2 - \frac{m}2 - \frac32 + \frac{|\beta|}2}
	ds
	\int_{\mathbb{R}^n}
		(-1)^l (-y)^\beta \mathcal{I}_{hk;n+m+3} (1,y)
	dy\\
	&+ U_{j;n+m+1}^S (t)
\end{split}
\]
\[
\begin{split}
%	\sum_{2l+|\beta|=n+m+1} \sum_{k,h=1}^n \frac{\partial_t^l \nabla^\beta R_k R_j G(t)}{l!\beta!} \int_0^\infty \int_{\mathbb{R}^n}
%		(-s)^l (-y)^\beta \bigl( \omega_{hk} u_h (s,y) - \sum_{p=3}^{m+2} \mathcal{I}_{hk;n+p}^\infty (s,y)\\
%		&\hspace{35mm} - \mathcal{I}_{hk;n+m+3}^\infty (1+s,y) \bigr)
%	dyds\\
	&-
	\sum_{k,h=1}^n \int_0^t \int_{\mathbb{R}^n}
		\biggl( R_k R_j G(t-s,x-y) - \sum_{2l+|\beta|=0}^{n+m+1} \frac{\partial_t^l \nabla^\beta R_k R_j G(t,x)}{l!\beta!} (-s)^l (-y)^\beta \biggr)
		\mathcal{I}_{hk;n+m+3} (s,y)
	dyds\\
%\end{split}
%\]
%\[
%\begin{split}
%	&+
%	\sum_{k,h= 1}^n \int_0^t \int_{\mathbb{R}^n}
%		\biggl( R_k R_j G(t-s,x-y) - \sum_{2l+|\beta|=0}^{n+m+1} \frac{\partial_t^l \nabla^\beta R_k R_j G(t,x)}{l!\beta!} (-s)^l (-y)^\beta \biggr)\\
%		&\hspace{5mm}\times
%		\biggl( \omega_{hk} u_h - \sum_{p=3}^{m+3} \mathcal{I}_{hk;n+p}^\infty \biggr) (s,y)
%	dyds\\
%\end{split}
%\]
%\[
%\begin{split}
	&+
	r_{1,n+m+1} (t)
\end{split}
\]
and
\[
\begin{split}
	r_{3,n+m} (t)
	=&
	\sum_{2l+|\beta|=1}^n \frac{2t^{-\frac{n}2-\frac{m}2-\frac12+l+\frac{|\beta|}2}}{n+m+1-2l-|\beta|} \frac{\partial_t^l \nabla^\beta G(t)}{l!\beta!} \sum_{h=1}^n \int_{\mathbb{R}^n}
		(-1)^l (-y)^\beta \mathcal{I}_{hj;n+m+3} (1,y)
	dy
	+ r_{3,n+m+1} (t),
\end{split}
\]
respectively.
Hence we conclude \eqref{kinou} and
\[
\begin{split}
	u_j (t)
	=&
	\sum_{k=1}^{n+m} \left( U_{j,k} + U_{j,k}^S \right) (t) + \sum_{k=1}^m \left( K_{j,n+k} + V_{j,n+k} + J_{j,n+k} \right) (t) + \sum_{k=3}^m \tilde{V}_{j,n+k} (t)\\
	&+ r_{0,n+m} (t) + \cdots + r_{3,n+m} (t)
\end{split}
\]
for $1 \le m \le n$.
Next, we show that
\begin{equation}\label{est-wtr}
	\| |x|^\mu r_{0,n+m} (t) \|_{L^q (\mathbb{R}^n)} + \cdots +\| |x|^\mu r_{3,n+m} (t) \|_{L^q (\mathbb{R}^n)}
	= o (t^{-\frac{n}2 (1-\frac1q) - \frac{n}2 - \frac{m}2 + \frac\mu{2}})
\end{equation}
as $t \to +\infty$ for $q = 1$ and $0 \le \mu \le n+m-1$, and for $1 < q \le \infty$ and $0 \le \mu \le n+m$.
For some positive function $R = R(t)$ with $R(t) = o (t^{1/2})$ as $t \to +\infty$, $r_{0,n+m}$ is split as
\[
\begin{split}
	r_{0,n+m} (t)
	=&
	-\sum_{|\alpha| = n + m + 2} \sum_{k=1}^n \int_{|y| \le \min (|x|/2,R(t))} \int_0^1
		\frac{\nabla^\alpha R_k (-\Delta)^{-1/2} G(t,x-\lambda y)}{\alpha!} \lambda^{n+m+1} (-y)^\alpha \omega_{0kj} (y)
	d\lambda dy\\
	&-\sum_{|\alpha| = n + m + 1} \sum_{k=1}^n \int_{R(t) < |y| \le |x|/2} \int_0^1
		\frac{\nabla^\alpha R_k (-\Delta)^{-1/2} G(t,x-\lambda y)}{\alpha!} \lambda^{n+m} (-y)^\alpha \omega_{0kj} (y)
	d\lambda dy\\
	&+\sum_{|\alpha|=n+m+1} \sum_{k=1}^n \frac{\nabla^\alpha R_k (-\Delta)^{-1/2} G(t,x)}{\alpha!} \int_{R(t) < |y| \le |x|/2} (-y)^\alpha \omega_{0kj} (y) dy\\
	&- \sum_{|\alpha| = 2} \sum_{k=1}^n \int_{|y| \ge |x|/2} \int_0^1 \frac{\nabla^\alpha R_k (-\Delta)^{-1/2} G(t,x-\lambda y)}{\alpha!} \lambda (-y)^\alpha \omega_{0kj} (y) d\lambda dy\\
	&+
	\sum_{|\alpha|=2}^{n+m+1} \sum_{k=1}^n \frac{\nabla^\alpha R_k (-\Delta)^{-1/2} G(t)}{\alpha!} \int_{|y| \ge |x|/2}
		(-y)^\alpha \omega_{0kj} (y)
	dy.
\end{split}
\]
Thus
\[
\begin{split}
	\bigl\| |x|^\mu r_{0,n+m} (t) \bigr\|_{L^1 (\mathbb{R}^n)}
	\le&
	C R(t) \sum_{|\alpha| = n+m+2} \sum_{k=1}^n \bigl\| |x|^\mu \nabla^\alpha R_k (-\Delta)^{-1/2} G (t) \bigr\|_{L^1 (\mathbb{R}^n)} \bigl\| |x|^{n+m+1} \omega_{0kj} \bigr\|_{L^1 (\mathbb{R}^n)}\\
	&+
	C \sum_{|\alpha|=n+m+1} \sum_{k=1}^n \bigl\| |x|^\mu \nabla^\alpha R_k (-\Delta)^{-1/2} G(t) \bigr\|_{L^1 (\mathbb{R}^n)} \int_{|y| > R(t)} \bigl| (-y)^\alpha \omega_{0kj} (y) \bigr| dy\\
%\end{split}
%\]
%\[
%\begin{split}
	&+
	C \sum_{|\alpha|=2}^{n+m+1} \sum_{k=1}^n \bigl\| |x|^{|\alpha|-2} \nabla^\alpha R_k (-\Delta)^{-1/2} G(t) \bigr\|_{L^1 (\mathbb{R}^n)} \bigl\| |x|^{\mu+2-|\alpha|} (-x)^\alpha \omega_{0kj} \bigr\|_{L^1 (\mathbb{R}^n)}
\end{split}
\]
for $0 \le \mu \le n+m-1$.
Similarly, since
\[
\begin{split}
	r_{0,n+m} (t)
	=&
	-\sum_{|\alpha| = n + m + 2} \sum_{k=1}^n \int_{|y| \le \min (|x|/2,R(t))} \int_0^1
		\frac{\nabla^\alpha R_k (-\Delta)^{-1/2} G(t,x-\lambda y)}{\alpha!} \lambda^{n+m+1} (-y)^\alpha \omega_{0kj} (y)
	d\lambda dy\\
	&- \sum_{|\alpha| = n+m+1} \sum_{k=1}^n \int_{R(t) < |y| \le |x|/2} \int_0^1 \frac{\nabla^\alpha R_k (-\Delta)^{-1/2} G(t,x-\lambda y)}{\alpha!} \lambda^{n+m} (-y)^\alpha \omega_{0kj} (y) d\lambda dy\\
	&+
	\sum_{|\alpha|= n+m+1} \sum_{k=1}^n \frac{\nabla^\alpha R_k (-\Delta)^{-1/2} G(t)}{\alpha!} \int_{R(t) < |y| \le |x|/2}
		(-y)^\alpha \omega_{0kj} (y)
	dy\\
	&- \sum_{|\alpha| = 1} \sum_{k=1}^n \int_{|y| \ge |x|/2} \int_0^1 \frac{\nabla^\alpha R_k (-\Delta)^{-1/2} G(t,x-\lambda y)}{\alpha!} (-y)^\alpha \omega_{0kj} (y) d\lambda dy\\
	&+
	\sum_{|\alpha|=1}^{n+m+1} \sum_{k=1}^n \frac{\nabla^\alpha R_k (-\Delta)^{-1/2} G(t)}{\alpha!} \int_{|y| \ge |x|/2}
		(-y)^\alpha \omega_{0kj} (y)
	dy,
\end{split}
\]
we have for $1 < q \le \infty$ and $0 \le \mu \le n+m$ that
\[
\begin{split}
	\bigl\| |x|^\mu r_{0,n+m} (t) \bigr\|_{L^q (\mathbb{R}^n)}
	\le&
	C R(t) \sum_{|\alpha| = n+m+2} \sum_{k=1}^n \bigl\| |x|^\mu \nabla^\alpha R_k (-\Delta)^{-1/2} G (t) \bigr\|_{L^q (\mathbb{R}^n)} \bigl\| |x|^{n+m+1} \omega_{0kj} \bigr\|_{L^1 (\mathbb{R}^n)}\\
	&+
	C \sum_{|\alpha|=n+m+1} \sum_{k=1}^n \bigl\| |x|^\mu \nabla^\alpha R_k (-\Delta)^{-1/2} G(t) \bigr\|_{L^q (\mathbb{R}^n)} \int_{|y| > R(t)} \bigl| |y|^{n+m+1} \omega_{0kj} (y) \bigr| dy\\
	&+
	C \sum_{|\alpha|=1}^{n+m+1} \sum_{k=1}^n \bigl\| |x|^{|\alpha|-1} \nabla^\alpha R_k (-\Delta)^{-1/2} G(t) \bigr\|_{L^q (\mathbb{R}^n)} \bigl\| |x|^{\mu+1-|\alpha|} (-x)^\alpha \omega_{0kj} \bigr\|_{L^1 (\mathbb{R}^n)}.
\end{split}
\]
Therefore $\| |x|^\mu r_{0,n+m} \|_{L^q (\mathbb{R}^n)} = o (t^{-\frac{n}2 (1-\frac1q)-\frac{n}2-\frac{m}2+\frac\mu{2}})$ as $t \to +\infty$ for $q = 1$ and $0 \le \mu \le n+m-1$, and for $1 < q \le \infty$ and $0 \le \mu \le n+m$.
Next we derive \eqref{est-wtr} for $r_{1,n+m},\ldots,r_{3,n+m}$.
We show this for large $\mu$ for a start.
We employ $Q_1,\ldots,Q_5$ defined by \eqref{Q}, then $r_{1,n+m} = r_{1,1,n+m} + \cdots + r_{1,5,n+m}$, where
\[
	r_{1,i,n+m} (t)
	=
	\left\{
	\begin{array}{r}
	\displaystyle
	- \sum_{k,h=1}^n \iint_{Q_i}
		\biggl( R_k R_j G(t-s,x-y) - \sum_{2l = 0}^{n+m} \frac{\partial_t^l R_k R_j G (t,x-y)}{l!} (-s)^l \biggr)\hspace{10mm}\\
		\displaystyle
		\times
		\biggl( \omega_{hk} u_h - \sum_{p=3}^{m+2} \mathcal{I}_{hk;n+p} \biggr) (s,y)
	dyds,\quad
	i = 1,2,3,\\
	\displaystyle
	- \sum_{2l=0}^{n+m} \sum_{k,h=1}^n \iint_{Q_i}
		\biggl( \frac{\partial_t^l R_k R_j G (t,x-y)}{l!} - \sum_{|\beta| = 0}^{n+m-2l} \frac{\partial_t^l \nabla^\beta R_k R_j G(t,x)}{l!\beta!} (-y)^\beta \biggr)\\
		\displaystyle
		\times
		(-s)^l \biggl( \omega_{hk} u_h - \sum_{p=3}^{m+2} \mathcal{I}_{hk;n+p} \biggr) (s,y)
	dyds,\quad
	i = 4,5.
	\end{array}
	\right.
\]
Let $N = \max \{ l \in \mathbb{Z}_+~ |~ 2l \le n+m \} + 1$, then
\[
\begin{split}
	r_{1,1,n+m} (t)
	=&
	- \sum_{k,h=1}^n \int_0^{t/2} \int_{|y| \ge |x|/2} \int_0^1
		\frac{\partial_t^N R_k R_j G (t-\lambda s,x-y)}{N!} \lambda^{N-1}\\
		&\hspace{5mm}\times
		(-s)^N \biggl( \omega_{hk} u_h - \sum_{p=3}^{m+2} \mathcal{I}_{hk;n+p} \biggr) (s,y)
	d\lambda dyds.
\end{split}
\]
Thus, by \eqref{bilin-asymp-wt},
\[
\begin{split}
	&\left\| |x|^\mu r_{1,1,n+m} (t) \right\|_{L^q (\mathbb{R}^n)}\\
	\le&
	C \sum_{k,h=1}^n \int_0^{t/2} \int_0^1
		\bigl\| \partial_t^N R_k R_j G (t-\lambda s) \bigr\|_{L^q (\mathbb{R}^n)} \lambda^{N-1}
		s^N \biggl\| |x|^\mu \biggl( \omega_{hk} u_h - \sum_{p=3}^{m+2} \mathcal{I}_{hk;n+p} \biggr) (s) \biggr\|_{L^1 (\mathbb{R}^n)}
	d\lambda ds\\
	\le&
	C \int_0^{t/2} \int_0^1
		(t-\lambda s)^{-\frac{n}2 (1-\frac1q) - N}
		\lambda^{N-1}
		s^{-\frac{n}2 - \frac{m}2 - 1 + \frac{\mu}2 + N} (1+s)^{-\frac12} L_m (s)
	d\lambda ds\\
	\le&
	C t^{-\frac{n}2 (1-\frac1q) - \frac{n}2 - \frac{m}2 - \frac12 + \frac\mu{2}} L_m (t).
\end{split}
\]
Similarly,
\[
\begin{split}
	r_{1,2,n+m} (t)
	=&
	- \sum_{k,h=1}^n \int_0^t \int_{|y| \le |x|/2} \int_0^1
		\frac{\partial_t^N R_k R_j G(t-\lambda s,x-y)}{N!} \lambda^{N-1}\\
		&\hspace{15mm}\times
		(-s)^N \biggl( \omega_{hk} u_h - \sum_{p=3}^{m+2} \mathcal{I}_{hk;n+p} \biggr) (s,y)
	d\lambda dyds.
\end{split}
\]
Since $2N-4 < \mu < 2N$ when $q=1$ and $n+m-2 < \mu \le n+m-1$,
\[
\begin{split}
	&\bigl\| |x|^\mu r_{1,2,n+m} (t) \bigr\|_{L^1 (\mathbb{R}^n)}\\
	\le&
	C \sum_{k,h=1}^n \int_0^t \int_0^1
		\bigl\| |x|^\mu \partial_t^N R_k R_j G(t-\lambda s) \bigr\|_{L^1 (\mathbb{R}^n)}
		\lambda^{N-1} s^N \biggl\| \biggl( \omega_{hk} u_h - \sum_{p=3}^{m+2} \mathcal{I}_{hk;n+p} \biggr) (s) \biggr\|_{L^1 (\mathbb{R}^n)}
	d\lambda ds\\
	\le&
	C \int_0^t \int_0^1
		(t-\lambda s)^{- N + \frac\mu{2}}
		\lambda^{N-1} s^{- \frac{n}2 - \frac{m}2 - 1 +N} (1+s)^{-\frac12} L_m (s)
	d\lambda ds
	\le
	C t^{- \frac{n}2 - \frac{m}2 - \frac12 + \frac\mu{2}} L_m (t).
\end{split}
\]
Here we remark that, since $-N+\frac\mu{2} > -2$,  $(t-\lambda s)^{-N+\frac\mu{2}}$ is integrable in $(s,\lambda) \in (0,t) \times (0,1)$.
Indeed, for $a > -2$,
\[
	\int_{t/2}^t \int_0^1
		(t-\lambda s)^a
	d\lambda ds
	=
	t^{1+a} \int_{1/2}^1 \int_0^1
		(1-\lambda s)^a
	d\lambda ds.
\]
When $a \neq -1$,
\[
\begin{split}
	\int_{1/2}^1 \int_0^1
		(1-\lambda s)^a
	d\lambda ds
	=&
	\int_{1/2}^1 \frac1{s} \int_{1-s}^1 \lambda^a d\lambda ds
	=
	\frac1{1+a} \int_{1/2}^1 \frac1s \left( 1 - (1-s)^{1+a} \right) ds\\
	=&
	\frac1{1+a} \biggl( \log2 - \int_{1/2}^1 \frac1s (1-s)^{1+a} ds \biggr).
\end{split}
\]
The last term satisfies
\[
	\biggl| \int_{1/2}^1 \frac1s (1-s)^{1+a} ds \biggr|
	\le
	-2 \int_{1/2}^1 (1-s)^{1+a} ds
	=
	\frac2{2+a} \left[ (1-s)^{2+a} \right]_{1/2}^1
	=
	\frac2{2+a} \left( 1 - \frac1{2^{2+a}} \right).
\]
When $a = -1$,
\[
\begin{split}
	\int_{1/2}^1 \int_0^1 (1-\lambda s)^{-1} d\lambda ds
	=
	\int_{1/2}^1 \frac1{s} \int_{1-s}^1 \frac{d\lambda}{\lambda} ds
	=
	- \int_{1/2}^1 \frac{\log (1-s)}s ds
	\le
	-2 \int_{1/2}^1 \log (1-s) ds
	=
	\log 2 + 1.
\end{split}
\]
For $1 < q \le \infty$, we choose some $q_1$ and $q_2$ with $1 + \frac1q = \frac1{q_1} + \frac1{q_2}$ and $1 < q_1 < \frac{n}{n-1}$, then for $n+m-1 \le \mu \le n+m$, we see that
\[
\begin{split}
	&\bigl\| |x|^\mu r_{1,2,n+m} (t) \bigr\|_{L^q (\mathbb{R}^n)}\\
	\le&
	C \sum_{k,h=1}^n \int_0^{t/2} \int_0^1
		\bigl\| |x|^\mu \partial_t^N R_k R_j G(t-\lambda s) \bigr\|_{L^q (\mathbb{R}^n)}
		\lambda^{N-1} s^N \biggl\| \biggl( \omega_{hk} u_h - \sum_{p=3}^{m+2} \mathcal{I}_{hk;n+p} \biggr) (s) \biggr\|_{L^1 (\mathbb{R}^n)}
	d\lambda ds\\
	&+
	C \sum_{k,h=1}^n \int_{t/2}^t \int_0^1
		\bigl\| |x|^\mu \partial_t^N R_k R_j G(t-\lambda s) \bigr\|_{L^{q_1} (\mathbb{R}^n)}
		\lambda^{N-1} s^N \biggl\| \biggl( \omega_{hk} u_h - \sum_{p=3}^{m+2} \mathcal{I}_{hk;n+p} \biggr) (s) \biggr\|_{L^{q_2} (\mathbb{R}^n)}
	d\lambda ds\\
	\le&
	C \int_0^{t/2} \int_0^1
		(t-\lambda s)^{ - \frac{n}2 (1-\frac1q) -N + \frac\mu{2}}
		\lambda^{N-1} s^{- \frac{n}2 - \frac{m}2 - 1 +N} (1+s)^{-\frac12} L_m (s)
	d\lambda ds\\
	&+
	C \int_{t/2}^t \int_0^1
		(t-\lambda s)^{-\frac{n}2 (1 -\frac1{q_1}) - N + \frac\mu{2}}
		\lambda^{N-1} s^{-\frac{n}2 (1-\frac1{q_2}) - \frac{n}2 - \frac{m}2 - 1 +N} (1+s)^{-\frac12} L_m (s) d\lambda ds\\
	\le&
	C t^{-\frac{n}2 (1-\frac1q) - \frac{n}2 - \frac{m}2 - \frac12 + \frac\mu{2}} L_m (t).
\end{split}
\]
Here we can choose $q_1$ such that $-\frac{n}2 (1 -\frac1{q_1}) - N + \frac\mu{2} > -2$, thus $(t-\lambda s)^{-\frac{n}2 (1 -\frac1{q_1}) - N + \frac\mu{2}}$ is integrable in $(s,\lambda) \in (t/2,t) \times (0,1)$.
Moreover
\[
\begin{split}
	r_{1,3,n+m} (t)
	=&
	- \sum_{k,h=1}^n \int_{t/2}^t \int_{|y| > |x|/2}
		\biggl( \int_0^1 \partial_t R_k R_j G(t-\lambda s,x-y) d\lambda (-s) - \sum_{2l=2}^{n+m} \frac{\partial_t^l R_k R_j G(t,x-y)}{l!} (-s)^l \biggr)\\
		&\hspace{15mm}\times
		\biggl( \omega_{hk} u_h - \sum_{p=3}^{m+2} \mathcal{I}_{hk;n+p} \biggr) (s,y)
	dyds.
\end{split}
\]
Hence
\[
\begin{split}
	&\bigl\| |x|^\mu r_{1,3,n+m} (t) \bigr\|_{L^1 (\mathbb{R}^n)}\\
	\le&
	C \int_{t/2}^t \biggl(
		s \int_0^1 \left\| \partial_t R_k R_j G(t-\lambda s) \right\|_{L^1 (\mathbb{R}^n)} d\lambda
		+
		\sum_{2l=2}^{n+m} s^l \bigl\| \partial_t^l R_k R_j G(t) \bigr\|_{L^1 (\mathbb{R}^n)}
	\biggr)\\
	&\hspace{15mm}\times
	\biggl\| |x|^\mu \biggl( \omega_{hk} u_h - \sum_{p=3}^{m+2} \mathcal{I}_{hk;n+p} \biggr) (s) \biggr\|_{L^1 (\mathbb{R}^n)}
	ds\\
	\le&
	C \int_{t/2}^t
		\biggl( s \int_0^1 (t-\lambda s)^{-1} d\lambda  + \sum_{2l=2}^n s^l t^{-l} \biggr)
		s^{-\frac{n}2 - \frac{m}2 - 1 + \frac\mu{2}} (1+s)^{-\frac12} L_m (s)
	ds
	=
	o \bigl( t^{- \frac{n}2 - \frac{m}2 + \frac\mu{2}} \bigr)
\end{split}
\]
as $t \to +\infty$.
For $1 < q \le \infty$, we choose $q_1$ and $q_2$ such that  $1+\frac1q = \frac1{q_1} + \frac1{q_2}$ and $1 - \frac2{n} < \frac1{q_1} < 1$, then
\[
\begin{split}
	&\bigl\| |x|^\mu r_{1,3,n+m} (t) \bigr\|_{L^q (\mathbb{R}^n)}
	\le
	\int_{t/2}^t
		\bigl\| R_k R_j G (t-s) \bigr\|_{L^{q_1} (\mathbb{R}^n)}
		\biggl\| |x|^\mu \biggl( \omega_{hk} u_h - \sum_{p=3}^{m+2} \mathcal{I}_{hk;n+p} \biggr) (s) \biggr\|_{L^{q_2} (\mathbb{R}^n)}
	ds\\
	&+
	C \sum_{2l=0}^{n+m} \bigl\| \partial_t^l R_k R_j G(t) \bigr\|_{L^q (\mathbb{R}^n)}
	\int_{t/2}^t s^l \biggl\| |x|^\mu  \biggl( \omega_{hk} u_h - \sum_{p=3}^{m+2} \mathcal{I}_{hk;n+p} \biggr) (s) \biggr\|_{L^1 (\mathbb{R}^n)} ds\\
	\le&
	C \int_{t/2}^t
		(t-s)^{-\frac{n}2 (1-\frac1{q_1})}
		s^{-\frac{n}2 (1-\frac1{q_2}) - \frac{n}2 - \frac{m}2-1+\frac\mu{2}} (1+s)^{-\frac12} L_m (s)
	ds\\
	&+
	C \sum_{2l=0}^{n+m} t^{-\frac{n}2(1-\frac1q) - l} \int_{t/2}^t s^{-\frac{n}2 - \frac{m}2 -1 +l + \frac\mu{2}} (1+s)^{-\frac12} L_m (s) ds\\
	\le&
	C t^{-\frac{n}2 (1-\frac1q) - \frac{n}2 -\frac{m}2+\frac\mu{2}} (1+t)^{-\frac12} L_m (t).
\end{split}
\]
From the Taylor theorem,
\[
\begin{split}
	r_{1,4,n+m} (t)
	=&
	- \sum_{2l = 0}^{n+m} \sum_{|\beta|=n+m-2l+1} \sum_{k,h=1}^n \int_0^t \int_{|y| \le |x|/2} \int_0^1
		\frac{\partial_t^l \nabla^\beta R_k R_j G(t,x-\lambda y)}{l!\beta!}\\
		&\hspace{5mm}\times
		(-s)^l (-y)^\beta \biggl( \omega_{hk} u_h - \sum_{p=3}^{m+2} \mathcal{I}_{hk;n+p} \biggr) (s,y)
	d\lambda dyds.
\end{split}
\]
Hence
\[
\begin{split}
	\bigl\| |x|^\mu r_{1,4,n+m} (t) \bigr\|_{L^q (\mathbb{R}^n)}
	\le&
	C \sum_{2l=0}^{n+m} \sum_{|\beta|=n+m-2l+1} \sum_{k,h=1}^n \int_0^t
		\bigl\| |x|^\mu \partial_t^l \nabla^\beta R_k R_j G(t) \bigr\|_{L^q (\mathbb{R}^n)}\\
		&\times
		s^l \biggl\| (-y)^\beta \biggl( \omega_{hk} u_h - \sum_{p=3}^{m+2} \mathcal{I}_{hk;n+p} \biggr) (s) \biggr\|_{L^1 (\mathbb{R}^n)}
	ds\\
	\le&
	C t^{-\frac{n}2 (1-\frac1q) - \frac{n}2 - \frac{m}2 - \frac12 + \frac\mu{2}} \int_0^t
		s^{-\frac12} (1+s)^{-\frac12} L_m (s)
	ds\\
	\le&
	C t^{-\frac{n}2 (1-\frac1q) - \frac{n}2 - \frac{m}2 - \frac12 + \frac\mu{2}} L_m (t) \log (2+t).
\end{split}
\]
The last term of $r_{1,n+m}$ is represented for $l_1 = 1$ and $2$ that
\begin{equation}\label{Taylor}
\begin{split}
	&r_{1,5,n+m} (t)\\
	&=
	- \sum_{k,h=1}^n \int_0^t \int_{|y| > |x|/2} \biggl(
		\sum_{|\beta|=l_1} \int_0^1
			\frac{\nabla^\beta R_k R_j G(t,x-\lambda y)}{\beta!} (-y)^\beta \lambda^{l-1}
		d\lambda
		-
		\sum_{|\beta|=l_1}^{n+m} \frac{\nabla^\beta R_k R_j G(t,x)}{\beta!} (-y)^\beta
		\biggr)\\
		&\hspace{5mm}\times
		\biggl( \omega_{hk} u_h - \sum_{p=3}^{m+2} \mathcal{I}_{hk;n+p} \biggr) (s,y)
	dyds\\
	&-
	\sum_{2l=2}^{n+m} \sum_{k,h=1}^n \int_0^t \int_{|y| > |x|/2}
		\biggl( \frac{\partial_t^l R_k R_j G (t,x-y)}{l!} - \sum_{|\beta| = 0}^{n+m-2l} \frac{\partial_t^l \nabla^\beta R_k R_j G(t,x)}{l!\beta!} (-y)^\beta \biggr)\\
		\displaystyle
		&\hspace{5mm}\times
		(-s)^l \biggl( \omega_{hk} u_h - \sum_{p=3}^{m+2} \mathcal{I}_{hk;n+p} \biggr) (s,y)
	dyds.
\end{split}
\end{equation}
We employ \eqref{Taylor} with $l_1 = 1$ for the case $1 < q \le \infty$ and $n+m-1 < \mu \le n+m$, then
\[
\begin{split}
	&\bigl\| |x|^\mu r_{1,5,n+m} (t) \bigr\|_{L^q (\mathbb{R}^n)}\\
	&\le
	C \sum_{|\beta|=1}^{n+m} \sum_{k,h=1}^n \int_0^t
		\bigl\| |x|^{|\beta|-1} \nabla^\beta R_k R_j G (t) \bigr\|_{L^q (\mathbb{R}^n)}
%		&\hspace{5mm}\times
		\biggl\| |x|^{\mu+1} \biggl( \omega_{hk} u_h - \sum_{p=3}^{m+2} \mathcal{I}_{hk;n+p} \biggr) (s) \biggr\|_{L^1 (\mathbb{R}^n)}
	ds\\
	&+C \sum_{2l=2}^{n+m} \sum_{|\beta|=0}^{n+m-2l} \sum_{k,h=1}^n \int_0^t
		\bigl\| |x|^{|\beta|} \partial_t^l \nabla^\beta R_k R_j G (t) \bigr\|_{L^q (\mathbb{R}^n)}
%		&\hspace{5mm}\times
		s^l \biggl\| |x|^\mu \biggl( \omega_{hk} u_h - \sum_{p=3}^{m+2} \mathcal{I}_{hk;n+p} \biggr) (s) \biggr\|_{L^1 (\mathbb{R}^n)}
	ds\\
	&\le
	C t^{-\frac{n}2 (1-\frac1q)-\frac12} \int_0^t
		s^{-\frac{n}2-\frac{m}2-\frac12+\frac\mu{2}} (1+s)^{-\frac12} L_m (s)
	ds\\
	&+
	C \sum_{2l=2}^{n+m} t^{-\frac{n}2 (1-\frac1q)-l} \int_0^t
		s^{-\frac{n}2-\frac{m}2 -1 + l + \frac\mu{2}} (1+s)^{-\frac12} L_m (s)
	ds
	=
	o \bigl( t^{-\frac{n}2 (1-\frac1q) - \frac{n}2 - \frac{m}2 + \frac\mu{2}} \bigr)
\end{split}
\]
as $t \to +\infty$.
For $n+m-2 < \mu \le n+m-1$, we use \eqref{Taylor} with $l_1 = 2$, then
\[
\begin{split}
	&\bigl\| |x|^\mu r_{1,5,n+m} (t) \bigr\|_{L^1 (\mathbb{R}^n)}\\
	&\le
	C \sum_{|\beta| =2}^{n+m} \sum_{k,h=1}^n \int_0^t
		\bigl\| |x|^{|\beta|-2} \nabla^\beta R_k R_j G(t) \bigr\|_{L^1 (\mathbb{R}^n)}
		\biggl\| |x|^{\mu+2} \biggl( \omega_{hk} u_h - \sum_{p=3}^{m+2} \mathcal{I}_{hk;n+p} \biggr) (s) \biggr\|_{L^1 (\mathbb{R}^n)}
	ds\\
	&+
	C \sum_{2l=2}^{n+m} \sum_{|\beta|=0}^{n+m-2l} \sum_{k,h=1}^n \int_0^t
		\bigl\| |x|^{|\beta|} \partial_t^l \nabla^\beta R_k R_j G(t) \bigr\|_{L^1 (\mathbb{R}^n)}
		s^l \biggl\| |x|^\mu \biggl( \omega_{hk} u_h - \sum_{p=3}^{m+2} \mathcal{I}_{hk;n+p} \biggr) (s) \biggr\|_{L^1 (\mathbb{R}^n)}
	ds\\
	&\le
	C t^{-1} \int_0^t s^{-\frac{n}2 - \frac{m}2 + \frac\mu{2}} (1+s)^{-\frac12} L_m (s) ds
	+
	C \sum_{2l=2}^{n+m} t^{-l} \int_0^t s^{-\frac{n}2 - \frac{m}2 - 1 + l + \frac\mu{2}} (1+s)^{-\frac12} L_m (s) ds\\
	&=
	o \bigl( t^{-\frac{n}2 - \frac{m}2 + \frac\mu{2}} \bigr)
\end{split}
\]
as $t \to +\infty$.
We estimate the second term of $r_{2,n+m}$ by the same way.
The treatment for the other terms of $r_{2,n+m}$ and $r_{3,n+m}$ is straightforward.
At the last we show \eqref{est-wtr} with $\mu = 0$.
The estimate for $r_{0,n+m}$ is already derived, and \eqref{bilin-asymp-wt} treats $r_{3,n+m}$.
For $N = \max \{ l \in \mathbb{Z}_+~ |~ 2l \le n+m \} + 1$,
\[
\begin{split}
	&r_{1,n+m} (t)\\
	=&
	- \sum_{k,h=1}^n \int_0^{t/2} \int_{\mathbb{R}^n} \int_0^1 \frac{\partial_t^N R_k R_j G(t-\lambda s,x-y)}{N!} \lambda^{N-1} d\lambda (-s)^N
	\biggl( \omega_{hk} u_h - \sum_{p=3}^{m+2} \mathcal{I}_{hk;n+p} \biggr) (s,y)
	dyds\\
%\end{split}
%\]
%\[
%\begin{split}
%	&+
%	\sum_{|\beta|=n+m+1} \sum_{k,h=1}^n \int_0^{t/2} \int_{|y| \le R(t)} \int_0^1
%		\frac{\nabla^\beta R_k R_j G(t,x-\lambda y)}{\beta!} \lambda^{n+m} d\lambda\\
%		&\hspace{15mm}
%		\times (-y)^\beta \biggl( \omega_{hk} u_h - \sum_{p=3}^{m+2} \mathcal{I}_{hk;n+p}^\infty \biggr) (s,y)
%	dyds\\
%	&+
%	\sum_{|\beta|=n+m} \sum_{k,h=1}^n \int_0^{t/2} \int_{|y| \ge R(t)}
%		\biggl( \int_0^1 \frac{\nabla^\beta R_k R_j G(t,x-\lambda y)}{\beta!} d\lambda - \frac{\nabla^\beta R_k R_j G(t,x)}{\beta!} \biggr)\\
%		&\hspace{15mm}
%		\times (-y)^\beta \biggl( \omega_{hk} u_h - \sum_{p=3}^{m+2} \mathcal{I}_{hk;n+p}^\infty \biggr) (s,y)
%	dyds\\
	&-
	\sum_{2l=0}^{n+m} \sum_{|\beta| = n+m-2l+1} \sum_{k,h=1}^n \int_0^{t/2} \int_{\mathbb{R}^n}
		\frac{\partial_t^l \nabla^\beta R_k R_j G(t,x-\lambda y)}{l!\beta!} \lambda^{n+m-2l} d\lambda\\
	&\hspace{15mm}\times
		(-s)^l (-y)^\beta \biggl( \omega_{hk} u_h - \sum_{p=3}^{m+2} \mathcal{I}_{hk;n+p} \biggr) (s,y)
	dyds\\
%\end{split}
%\]
%\[
%\begin{split}
	&-
	\int_{t/2}^t \int_{\mathbb{R}^n} \int_0^1
		(-y) \cdot \nabla R_k R_j G(t-s,x-\lambda y) d\lambda
		\biggl( \omega_{hk} u_h - \sum_{p=3}^{m+2} \mathcal{I}_{hk;n+p} \biggr) (s,y)
	dyds\\
	&+
	\sum_{2l+|\beta|=1}^{n+m} \frac{\partial_t^l \nabla^\beta R_k R_j G(t,x)}{l!\beta!}
	\int_{t/2}^t \int_{\mathbb{R}^n}
		(-s)^l (-y)^\beta
		\biggl( \omega_{hk} u_h - \sum_{p=3}^{m+2} \mathcal{I}_{hk;n+p} \biggr) (s,y)
	dyds.
\end{split}
\]
Hence
\[
\begin{split}
	\bigl\| r_{1,n+m} (t) \bigr\|_{L^q (\mathbb{R}^n)}
	\le&
	C \int_0^{t/2}
		(t-s)^{-\frac{n}2 (1-\frac1q)-N}
		s^{-\frac{n}2-\frac{m}2-1+N} (1+s)^{-\frac12} L_m (s)
	ds\\
	&+
	C t^{-\frac{n}2 (1-\frac1q) - \frac{n}2 - \frac{m}2 - \frac12}
	\int_0^{t/2}
		s^{-\frac12} (1+s)^{-\frac12} L_m (s)
	ds\\
	&+
	C \int_{t/2}^t
		(t-s)^{-\frac12}
		s^{-\frac{n}2 (1-\frac1q)-\frac{n}2 -\frac{m}2-\frac12} (1+s)^{-\frac12} L_m (s)
	ds\\
	&+
	C \sum_{2l+|\beta|=1}^{n+m} t^{-\frac{n}2 (1-\frac1q)-l-\frac{|\beta|}2}
	\int_{t/2}^t
		s^{-\frac{n}2-\frac{m}2-1+l+\frac{|\beta|}2} (1+s)^{-\frac12} L_m (s)
	ds\\
	\le&
	C t^{-\frac{n}2 (1-\frac1q) - \frac{n}2 - \frac{m}2 - \frac12} L_m (t) \log (2+t)
\end{split}
\]
for $1 \le q \le \infty$.
We apply the similar estimate to the second term of $r_{2,n+m}$, then \eqref{bilin-asymp-wt} gives that
\[
\begin{split}
	\bigl\| r_{2,n+m} (t) \bigr\|_{L^q (\mathbb{R})}
	\le&
	C t^{-\frac{n}2 (1-\frac1q)-\frac{n}2 - \frac{m}2} \int_t^\infty
		s^{-\frac12} (1+s)^{-1} L_m (s)
	ds
	+
	C t^{-\frac{n}2 (1-\frac1q) - \frac{n}2 - \frac{m}2 - \frac12} L_m (t) \log (2+t)\\
	&+
	C \sum_{2l+|\beta|=n+1}^{n+m-1} t^{-\frac{n}2 (1-\frac1q)-l-\frac{|\beta|}2} \int_t^\infty
		s^{-\frac{n}2-\frac{m}2-1+l+\frac{|\beta|}2} (1+s)^{-\frac12} L_m (s)
	ds\\
	&+
	C t^{-\frac{n}2 (1-\frac1q) - \frac{n}2 - \frac{m}2 + \frac12} \sum_{2l+|\beta|=n+m-1}
	\int_t^\infty s^l \left( (1+s)^{-\frac{n}2 - \frac{m}2 - \frac12 + \frac{|\beta|}2} - s^{-\frac{n}2 - \frac{m}2 - \frac12 + \frac{|\beta|}2} \right) ds\\
	\le&
	C t^{-\frac{n}2 (1-\frac1q) - \frac{n}2 - \frac{m}2 - \frac12} L_m (t) \log (2+t)
\end{split}
\]
for $1 \le q \le \infty$.
Therefore we obtain \eqref{est-wtr} with $\mu = 0$.
The H\"older inequality completes the proof.\hfill$\square$\\

Next, we show Theorem \ref{thm-main}.\\

\paragraph{\it Proof of Theorem \ref{thm-main}.}
From \eqref{s0}, we expand $u$ as
\[
	u_j (t)
	=
	\sum_{m=1}^n \left( U_{j;m} + U_{j;m}^T \right) (t)
	+
	r_{0,n} (t) + r_{1,n} (t) + r_{2,n} (t) + r_{3,n} (t) + r_{4,n} (t),
\]
where $r_{0,n},\ldots,r_{3,n}$ are defined as in the proof of Theorem \ref{thm-st}, and 
\[
\begin{split}
	r_{4,n} (t)
	=&
	\sum_{m=1}^n \bigl( U_{j;m}^S - U_{j;m}^T \bigr)(t)
	=
	\sum_{2l+|\beta|=1}^n \sum_{k,h=1}^n \frac{\partial_t^l \nabla^\beta R_k R_j G(t)}{l!\beta!} \int_t^\infty \int_{\mathbb{R}^n}
		(-s)^l (-y)^\beta (\omega_{hk} u_h) (s,y)
	dyds.
\end{split}
\]
Moreover, we expand $r_{1,n}$ and $r_{4,n}$, then, from \eqref{int-I},
\[
\begin{split}
	r_{1,n} (t)
	=&
	- \sum_{k,h = 1}^n \int_0^t \int_{\mathbb{R}^n}
		\biggl(
			R_k R_j G(t-s,x-y) - \sum_{2l + |\beta| = 0}^{n+1} \frac{\partial_t^l \nabla^\beta R_k R_j G (t,x)}{l! \beta!} (-s)^l (-y)^\beta
		\biggr)
		\mathcal{I}_{hk;n+3} (s, y)
	dyds\\
	&
	+ U_{j;n+1}^T (t)\\
%	-
%	\sum_{k,h=1}^n
%	\sum_{2l + |\beta| = n+1} \frac{\partial_t^l \nabla^\beta R_k R_j G (t,x)}{l!\beta!} \int_0^\infty \int_{\mathbb{R}^n}
%		(-s)^l (-y)^\beta \left( \omega_{hk} u_h (s,y) - \mathcal{I}_{hk;n+3} ({\color{red} 1+s},y) \right)
%	dyds\\
	&-
	\sum_{k,h=1}^n \sum_{2l + |\beta| = n+1} \frac{\partial_t^l \nabla^\beta R_k R_j G(t,x)}{l!\beta!}
	\int_0^t s^l (1+s)^{-\frac{n}2 - \frac32 + \frac{|\beta|}2} ds
	\int_{\mathbb{R}^n}
		(-1)^l (-y)^\beta \mathcal{I}_{hk;n+3} (1,y)
	dy\\
	&+r_{1,n+1} (t) + r_{1,n+1}^T (t),\\
	r_{4,n} (t)
	=&
	\sum_{2l+|\beta| = 1}^n \sum_{k,h=1}^n \frac{\partial_t^l \nabla^\beta R_k R_j G(t)}{l!\beta!} \int_t^\infty \int_{\mathbb{R}^n}
		(-s)^l (-y)^\beta \mathcal{I}_{hk;n+3} (s,y)
	dyds
	+ r_{4,n+1} (t),
\end{split}
\]
where
\[
\begin{split}
	r_{1,n+1}^T (t)
	&=
	U_{j;n+1}^S (t) - U_{j;n+1}^T (t)\\
	&=
	\sum_{2l + |\beta| = n+1} \sum_{k,h=1}^n \frac{\partial_t^l \nabla^\beta R_k R_j G(t,x)}{l!\beta!} \int_t^\infty \int_{\mathbb{R}^n}
		(-s)^l (-y)^\beta
%		&\hspace{15mm}\times
		\left( \omega_{hk} u_h (s,y) - \mathcal{I}_{hk;n+3} (1+s, y) \right)
	dyds,\\
	r_{4,n+1} (t)
	&=
	\sum_{2l+|\beta|=1}^n \sum_{k,h=1}^n \frac{\partial_t^l \nabla^\beta R_k R_j G(t)}{l!\beta!} \int_t^\infty \int_{\mathbb{R}^n}
		(-s)^l (-y)^\beta \left( \omega_{hk} u_h(s,y) - \mathcal{I}_{hk;n+3} (s,y) \right)
	dyds.
\end{split}
\]
The other terms $r_{0,n}, r_{2,n}$ and $r_{3,n}$ are treated as in the proof of Theorem \ref{thm-st}, hence we see for $1 \le m \le n$ that
\begin{equation}\label{exp-t}
\begin{split}
	u_j (t)
	=&
	\sum_{k=1}^{n+m} \left( U_{j,k} + U_{j,k}^T \right) (t)
	+
	\sum_{k=1}^m \left( K_{j,n+k} + V_{j,n+k} + V_{j,n+k}^T + J_{j,n+k} \right) (t)
	+
	\sum_{k=3}^m \tilde{V}_{j,n+k}^T (t)\\
	&+
	r_{0,n+m} + r_{1,n+m} (t) + r_{1,n+m}^T (t) + r_{2,n+m} (t) + r_{3,n+m} (t) + r_{4,n+m} (t),
\end{split}
\end{equation}
where
\[
\begin{split}
	r_{1,n+m}^T (t)
	=&
%	\sum_{2l+|\beta| = n+m} \sum_{k,h=1}^n \frac{\partial_t^l \nabla^\beta R_k R_j G(t)}{l!\beta!} \int_t^\infty \int_{\mathbb{R}^n}
%		(-s)^l (-y)^\beta \bigl( \omega_{hk} u_h (s,y) - \sum_{p=3}^{m+1} \mathcal{I}_{hk;n+p} (s,y)\\
%		&\hspace{15mm}-\mathcal{I}_{hk;n+m+2} (1+s,y) \bigr)
%	dyds\\
	U_{j;n+m}^S (t) - U_{j;n+m}^T (t)\\
%\end{split}
%\]
%\[
%\begin{split}
%\end{split}
%\]
%\[
%\begin{split}
	&+
	\sum_{2l+|\beta|=n+1}^{n+m-1} \sum_{k,h=1}^n \frac{\partial_t^l \nabla^\beta R_k R_j G(t)}{l!\beta!} \int_t^\infty \int_{\mathbb{R}^n}
		(-s)^l (-y)^\beta \biggl( \omega_{hk} u_h - \sum_{p=3}^{m+2} \mathcal{I}_{hk;n+p} \biggr) (s,y)
	dyds\\
%\end{split}
%\]
%\[
%\begin{split}
	&-
	\sum_{2l+|\beta|=n+m-1} \sum_{k,h=1}^n \frac{\partial_t^l \nabla^\beta R_k R_j G(t)}{l!\beta!}
	\int_t^\infty s^l \bigl( (1+s)^{-\frac{n}2 - \frac{m}2 - \frac12 + \frac{|\beta|}2} - s^{-\frac{n}2 - \frac{m}2 - \frac12 + \frac{|\beta|}2} \bigr) ds\\
	&\hspace{5mm}\times
	\int_{\mathbb{R}^n} (-1)^l (-y)^\beta \mathcal{I}_{hk;n+m+1} (1,y) dy,\\
%\end{split}
%\]
%\[
%\begin{split}
%	%
	r_{4,n+m} (t)
	=&
	\sum_{2l+|\beta|=1}^n \sum_{k,h=1}^n \frac{\partial_t^l \nabla^\beta R_k R_j G(t)}{l!\beta!} \int_t^\infty \int_{\mathbb{R}^n}
		(-s)^l (-y)^\beta \biggl( \omega_{hk} u_h - \sum_{p=3}^{m+2} \mathcal{I}_{hk;n+p} \biggr) (s,y)
	dyds.
\end{split}
\]
Indeed, for $1 \le m \le n-1$, the first term of $r_{1,n+m}^T$ is split to
\[
\begin{split}
	&U_{j;n+m}^S (t) - U_{j;n+m}^T (t)\\
	&=\sum_{2l+|\beta| = n+m} \sum_{k,h=1}^n \frac{\partial_t^l \nabla^\beta R_k R_j G(t)}{l!\beta!} \int_t^\infty \int_{\mathbb{R}^n}
		(-s)^l (-y)^\beta \bigl( \omega_{hk} u_h (s,y) - \sum_{p=3}^{m+1} \mathcal{I}_{hk;n+p} (s,y)\\
		&\hspace{15mm}-\mathcal{I}_{hk;n+m+2} (1+s,y) \bigr)
	dyds\\
%\end{split}
%\]
%\[
%\begin{split}
	&=
	\sum_{2l+|\beta|=n+m} \sum_{k,h=1}^n
	\frac{2t^{- \frac12} \partial_t^l \nabla^\beta R_k R_j G (t)}{l!\beta!} \int_{\mathbb{R}^n} (-1)^l (-y)^\beta \mathcal{I}_{hk;n+m+3} (1,y) dy\\
	&-
	\sum_{2l+|\beta|=n+m} \sum_{k,h=1}^n \frac{\partial_t^l \nabla^\beta R_k R_j G(t)}{l!\beta!}
	\int_t^\infty
		s^l \left( (1+s)^{-\frac{n}2 - \frac{m}2 -1 + \frac{|\beta|}2} - s^{-\frac{n}2 - \frac{m}2 -1 + \frac{|\beta|}2} \right)
	ds
	\int_{\mathbb{R}^n}
		\mathcal{I}_{hk;n+m+2} (1,y)
	dy\\
	&+
	\sum_{2l+|\beta|=n+m} \sum_{k,h=1}^n \frac{\partial_t^l \nabla^\beta R_k R_j G(t)}{l!\beta!} \int_t^\infty \int_{\mathbb{R}^n}
		(-s)^l (-y)^\beta \biggl( \omega_{hk} u_h - \sum_{p=3}^{m+3} \mathcal{I}_{hk;n+p} \biggr) (s,y)
	dyds.
\end{split}
\]
We split the second term of $r_{1,n+m}^T$, then
\[
\begin{split}
	&\sum_{2l+|\beta|=n+1}^{n+m-1} \sum_{k,h=1}^n \frac{\partial_t^l \nabla^\beta R_k R_j G(t)}{l!\beta!} \int_t^\infty \int_{\mathbb{R}^n}
		(-s)^l (-y)^\beta \biggl( \omega_{hk} u_h - \sum_{p=3}^{m+2} \mathcal{I}_{hk;n+p} \biggr) (s,y)
	dyds\\
%\end{split}
%\]
%\[
%\begin{split}
	=&
	\sum_{2l+|\beta|=n+1}^{n+m-1} \sum_{k,h=1}^n \frac{2t^{-\frac{n}2 - \frac{m}2 - \frac12 + l + \frac{|\beta|}2}}{n+m+1 - 2l - |\beta|} \frac{\partial_t^l \nabla^\beta R_k R_j G(t)}{l!\beta!} \int_{\mathbb{R}^n}
		(-1)^l (-y)^\beta \mathcal{I}_{hk;n+m+3} (1,y)
	dy\\
	&+
	\sum_{2l+|\beta|=n+1}^{n+m-1} \sum_{k,h=1}^n \frac{\partial_t^l \nabla^\beta R_k R_j G(t)}{l!\beta!} \int_t^\infty \int_{\mathbb{R}^n}
		(-s)^l (-y)^\beta \biggl( \omega_{hk} u_h - \sum_{jp=3}^{m+3} \mathcal{I}_{hk;n+p} \biggr) (s,y)
	dyds.
\end{split}
\]
The last term of $r_{1,n+m}^T$ is $\tilde{V}_{j;n+m+1}^T$.
Moreover, from \eqref{int-I},
\[
\begin{split}
	r_{4,n+m} (t)
	=&
	\sum_{2l+|\beta|=1}^n \sum_{k,h=1}^n \frac{2t^{-\frac{n}2-\frac{m}2-\frac12+l+\frac{|\beta|}2}}{n+m+1-2l-|\beta|} \frac{\partial_t^l \nabla^\beta R_k R_j G(t)}{l!\beta!} \int_{\mathbb{R}^n}
		(-1)^l (-y)^\beta \mathcal{I}_{hk;n+m+3} (1,y)
	dy\\
	&+ r_{4,n+m+1} (t).
\end{split}
\]
The other terms are treated in the similar way as in the proof of Theorem \ref{thm-st}.
Thus
\[
\begin{split}
	&r_{0,n+m} + r_{1,n+m}^T + r_{1,n+m} + \cdots + r_{4,n+m}\\
	=&
	K_{j;n+m+1} + U_{j;n+m+1} + U_{j;n+m+1}^T + V_{j;n+m+1} + V_{j;n+m+1}^T + \tilde{V}_{j;n+m+1}^T + J_{j;n+m+1}\\
	&+ r_{0,n+m+1} + r_{1,n+m+1}^T + r_{1,n+m+1} + \cdots  + r_{4,n+m+1}
\end{split}
\]
for $1 \le m \le n-1$, and \eqref{exp-t} holds.
Similar estimates for $r_{2,n+m}$ and $r_{3,n+m}$ as in the proof of Theorem \ref{thm-st} provide that
\[
	\bigl\| r_{1,n+m}^T (t) \bigr\|_{L^q (\mathbb{R}^n)} + \bigl\| r_{4,n+m} (t) \bigr\|_{L^q (\mathbb{R}^n)}
	=
	o \bigl( t^{-\frac{n}2 (1-\frac1q) - \frac{n}2 - \frac{m}2} \bigr)
\]
as $t \to + \infty$ for $1 \le q \le \infty$.
We already treated the other terms $r_{0,n+m},\ldots,r_{3,n+m}$ in the proof of Theorem \ref{thm-st}.
Therefore we complete the proof.\hfill$\square$

\end{document}